\documentclass[12pt]{amsart}
\usepackage[latin1]{inputenc}
\usepackage[active]{srcltx}
\usepackage{fullpage}
\usepackage{amsfonts,enumerate}
\usepackage[mathscr]{euscript}
\usepackage{epsfig}
\usepackage{latexsym}
\usepackage[all]{xy}
\usepackage{amssymb}
\usepackage{amsthm}
\usepackage{amsmath}
\usepackage{amsxtra}
\usepackage{mathtools}
\usepackage{xcolor}
\usepackage[colorlinks]{hyperref}
\usepackage{fancyvrb}
\hypersetup{citecolor=blue}

\usepackage[textwidth=2.5cm]{todonotes}

.tfm
.tfm

\theoremstyle{plain}
\newtheorem{prop}{Proposition}[section]
\newtheorem{thm}[prop]{Theorem}
\newtheorem{coro}[prop]{Corollary}

\newtheorem{lemma}[prop]{Lemma}

\newtheorem*{thm*}{Theorem}
\newtheorem*{lemma*}{Lemma}
\newtheorem*{prop*}{Proposition}

\newtheorem*{thmB}{Theorem~\ref{thm:nonexistence}}
\newtheorem*{thmC}{Theorem~\ref{thm:asymptotic}}

\newtheorem*{thmE}{Theorem~\ref{thm:asymptotic2}}

\theoremstyle{definition}
\newtheorem*{defi}{Definition}

\theoremstyle{remark}

\newtheorem*{exm}{Example}
\newtheorem{remark}{Remark}

\numberwithin{table}{section}

\DeclareMathOperator{\Frob}{Frob}

\DeclareMathOperator{\End}{End}
\DeclareMathOperator{\Hom}{Hom}

\DeclareMathOperator{\Gal}{Gal}
\DeclareMathOperator{\Aut}{Aut}
\DeclareMathOperator{\Ind}{Ind}
\DeclareMathOperator{\HH}{H}

\DeclareSymbolFont{bbold}{U}{bbold}{m}{n}
\DeclareSymbolFontAlphabet{\mathbbold}{bbold}
\DeclareMathOperator{\identity}{\mathbbold{1}}

\newcommand{\F}{\mathbb F}

\newcommand{\Om}{{\mathscr{O}}}

\newcommand{\Disc}{\Delta}

\newcommand{\GL}{{\rm GL}}

\newcommand{\PGL}{{\rm PGL}}

\newcommand{\E}{E_{(a,b,c)}}
\newcommand{\Et}{\widetilde{E}_{(a,b,c)}}

\def\AA{\mathbb A}

\def\ZZ{\mathbb Z}

\def\QQ{\mathbb Q}
\def\II{\mathbb I}

\def\AA{\mathbb A}
\def\CC{\mathbb C}

\def\ord{t}
\def\cond{r}
\def\ordkappa{M}

\def\<#1>{{\left\langle{#1}\right\rangle}}
\def\abs#1{{\left|{#1}\right|}}

\def\Z{{\mathbb Z}}             
\def\Q{{\mathbb Q}}             

\def\id#1{{\mathfrak{#1}}}      

\DeclareMathOperator{\norm}{{\mathscr N}}

\DeclareMathOperator{\trace}{{\mathrm{Tr}}}

\newcommand{\lmfdbecnf}[4]{\href{http://www.lmfdb.org/EllipticCurve/#1/#2/#3/#4}{{\text{\rm#1-#2-#3#4}}}}

\let\kro\dkro

\begin{document}
	
\title[Endomorphism Algebras]{On endomorphism algebras of $\GL_2$-type
  abelian varieties and Diophantine applications}
	
\author{Franco Golfieri Madriaga}
	
\address{Center for Research and Development in Mathematics and
  Applications (CIDMA), Department of Mathematics, University of
  Aveiro, 3810-193 Aveiro, Portugal}
	
\email{francogolfieri@ua.pt}
	
\author{Ariel Pacetti}
	
\address{Center for Research and Development in Mathematics and
  Applications (CIDMA), Department of Mathematics, University of
  Aveiro, 3810-193 Aveiro, Portugal} \email{apacetti@ua.pt}
	
\author{Lucas Villagra Torcomian}
	
\address{Department of Mathematics, Simon Fraser University, Burnaby, BC V5A 1S6, Canada.}  \email{lvillagr@sfu.ca}


\begin{abstract}
  Let $f$ and $g$ be two different newforms without complex
  multiplication having the same coefficient field.  The main result
  of the present article proves that an isomorphism between the
  residual Galois representations attached to $f$ and to $g$ for a
  large prime $p$ (depending only on $g$) implies that the
  endomorphism algebra of the abelian variety $A_f$, attached to $f$
  by the Eichler-Shimura construction, (after tensoring with $\Q$) is
  a subalgebra of the endomorphism algebra of the abelian variety
  $A_g$ attached to $g$.  This implies important relations between
  their building blocks. A non-trivial application of our result is
  that for all prime numbers $d$ congruent to $3$ modulo $8$
  satisfying that the class number of $\Q(\sqrt{-d})$ is prime to $3$,
  the equation $x^4+dy^2 =z^p$ has no non-trivial primitive solutions
  when $p$ is large enough. We prove a similar result for the equation
  $x^2+dy^6=z^p$.
\end{abstract}
	
\maketitle

\section*{Introduction}
	
Let $\tilde{N}$ be a positive integer and let $\varepsilon$ be a
Dirichlet character of conductor dividing $\tilde{N}$. Let
$g \in S_2(\Gamma_0(\tilde{N}), \varepsilon)$
be a newform of weight 2, level $\tilde{N}$ and Nebentypus
$\varepsilon$. Let $K_g$ denote the coefficient field of the newform
$g$ (i.e.\ the minimum number field containing all the Fourier
coefficients $a_n(g)$ of $g$). For $\id{p}$ a prime ideal of $K_g$, we
denote $\rho_{g,\id{p}}$ the Galois representation attached to $g$ (by
Eichler and Shimura). After choosing an appropriate basis for the
underlying vector space, one can always assume that the representation
has coefficients in the ring of integers of the completion of $K_g$ at
$\id{p}$, so it makes sense to consider its reduction
$\overline{\rho}_{g,\id{p}}$.
	
Let $N$ be a divisor of $\tilde{N}$ and let
$f \in S_2(\Gamma_0(N), \varepsilon)$ be another newform satisfying
the following conditions:
	
\begin{enumerate}
\item The coefficient field $K_f$ of $f$ matches the coefficient field
  $K_g$ of $g$.
		
\item There exists a prime $p$ such that the semisimplification of the
  residual Galois representations $\bar{\rho}{_{f,\id{p}}}$ and
  $\bar{\rho}{_{g,\id{p}}}$ are isomorphic for some prime ideal
  $\id{p}$ of $K_g$ dividing $p$.
\end{enumerate}
By Eichler-Shimura's construction, there exist abelian varieties $A_f$
and $A_g$ defined over $\Q$ of dimension $[K_f:\Q]$ attached to each
of the eigenforms with the property that
\[
  L(A_f,s) = \prod_{\sigma\in \Hom(K_f,\CC)}L(\sigma(f),s),
\]
and a similar relation for $g$. Let $L/\Q$ be a field extension, and
denote by $\End_L(A_f)$ the ring of endomorphisms of $A_f$ defined
over $L$.
	
\vspace{5pt}
	
\noindent {\bf Question 1:} Is there a relation between
$\End_L(A_f)\otimes \Q$ and $\End_L(A_g)\otimes \Q$?
	
\vspace{5pt}
	
By a result of Hecke (\cite{MR0371577}, p.811, Satz 1 and Satz 2), if
the Fourier coefficients $a_n(f) = a_n(g)$ for sufficiently many small
values of $n$, then the two forms coincide. A congruence between
$f$ and $g$ modulo a large prime ideal $\id{p}$ implies (as will be
explained in the proof of Theorem~\ref{thm:main}) that their first
Fourier coefficients must be equal (due to the Ramanujan-Petersson
bound), proving the existence of a constant $M$ (depending on the
level of $f$ and the level of $g$) such that if $\id{p}$ is a prime
ideal whose norm is larger than $M$ then an isomorphism (of the
semisimplifications of)
$\bar{\rho}{_{f,\id{p}}} \simeq \bar{\rho}{_{g,\id{p}}}$ implies that
$f=g$. In particular their endomorphism algebras are isomorphic. The
non-trivial problem is whether given the form $g$, there exists a
constant $M_g$ (depending either on the form $g$ or on its level) such
that if $p > M_g$ then the endomorphism algebras of both varieties are
related for any form $f$.
	
One of the main results of the present article
(Theorem~\ref{thm:main}) is to provide a positive answer to this
second problem when $g$ does not have complex multiplication and
$\tilde{N}/N$ is square-free, namely we prove the existence of a
constant $M_g$ such that if $p > M_g$, then for any newform $f$
without complex multiplication and satisfying conditions $(1)$ and
$(2)$, there exists an injective morphism
$\psi:\End_L(A_f)\otimes \Q \to \End_L(A_g)\otimes \Q$.
	
Our method could be used to prove a similar result for modular forms
with complex multiplication; however, because of the applications we
have in mind, in the present article we restrict to the case of
modular forms without complex multiplication. Similarly, the
hypothesis $\tilde{N}/N$ square-free can be removed if one makes a
more detailed study of local types, but it simplifies some proofs and
is satisfied in our applications.
	
The proof is based on results of Ribet on the splitting of a modular
$\GL_2$-type abelian variety $A_f$ over a number field $L$ (as
developed in \cite{ribet1980twists,MR1212980,MR1299737}). In such
articles, the author constructs endomorphisms of $A_f$ in terms of
properties of Fourier coefficients of the form $f$. Then to provide an
answer to Question 1 it is enough to relate the Fourier coefficients
of $f$ to those of $g$.
	
The proof of Theorem~\ref{thm:main} consists on proving that the
congruence between
the newforms $f$ and $g$ implies that all inner twists of $g$ are also
inner twists of $f$. The constructed morphism is not just a morphism
of $\Q$-algebras, but also a morphism of $\Gal(L/\Q)$-modules. This
provides a relation between the splitting (up to isogeny) of the
abelian variety $A_f$ and that of $A_g$ over any field extension
$L/\Q$.
	
\vspace{5pt}
	
Here is, in our opinion, an interesting application of our main result
to the study of Diophantine equations (the original motivation for
this article) following the modular approach. Consider the equation
\begin{equation}
  \label{eq:42p}
  x^4 + dy^2 = z^p,
\end{equation}
for a fixed positive square-free integer $d$. A solution $(a,b,c)$ of
a Fermat-type equation such as~(\ref{eq:42p}) is called
\textit{primitive} if $\gcd(a,b,c)=1$, and is also said to be
\textit{trivial} if $abc=0$. To a putative primitive solution
$(a,b,c)$ one attaches an elliptic curve $\E$ defined over the
imaginary quadratic field $K\vcentcolon =\QQ(\sqrt{-d})$ given by the
equation
\begin{equation}
  \label{eq:curve-E}
  \E: y^2=x^3+4ax^2+2(a^2+\sqrt{-d}b)x, 
\end{equation}
whose discriminant equals $512(a^2+b\sqrt{-d})c^p$. The curve $\E$ has
additive/multiplicative reduction at the primes dividing $2$,
multiplicative reduction at all odd primes dividing $c$ and good
reduction at all other primes (our primitivity assumption implies that
$\gcd(c,d)=1$). The curve $\E$ turns out to be a $\QQ$-curve hence in
particular (by a result of Ribet in \cite{MR2058653}, see the
paragraph after Theorem 6.3), there exists a Hecke character
$\varkappa: \Gal_K \rightarrow \overline{\QQ}^{\times}$, where
$\Gal_K$ denotes the absolute Galois group of $K$, such that the
twisted representation
\[
  \rho \vcentcolon = \rho_{\E,p} \otimes \varkappa : \Gal_K
  \rightarrow \GL_2(\overline{\QQ_p})
\]
extends to a representation $\tilde{\rho}$ of the whole absolute
Galois group $\Gal_{\QQ}$. If the residual representation of
$\tilde{\rho}$ is reducible, then $\tilde{\rho}$ is modular by
\cite[Theorem 1.0.2]{MR4467307}. Otherwise, the residual
representation of $\tilde{\rho}$ is modular by Serre's conjecture
(proven in \cite{MR2551763,MR2551764}). Then $\tilde{\rho}$ itself is
modular by \cite{Kisinann} (theorem in the second page), i.e.\ there
exist a modular form $f \in S_2(\Gamma_0(N),\varepsilon)$, where $N$
is the conductor of the representation $\tilde{\rho}$ and
$\varepsilon$ its Nebentypus (an explicit formula for $N$ and
$\varepsilon$ is given in \cite[Theorem 4.2]{PT}) and a prime ideal
$\frak{p}$ in the coefficient field $K_f$ of $f$ such that
$\tilde{\rho} \simeq \rho_{f, \frak{p}}$.
	
The curve $\E$ has multiplicative reduction at all odd primes dividing
$c$ (since the solution is primitive, if a prime $q$ divides $c$, it
cannot divide $d$), so by a result of Hellegouarch (\cite[Theorem
6.5.1]{MR1858559}) together with Ribet's lowering the level result,
there exists a newform $g$ of level $\tilde{N}$ (only divisible by
primes dividing $2d$) congruent to $f$ modulo $p$. Then one is led to
compute the space $S_2(\Gamma_0(\tilde{N}), \varepsilon)$ and to prove
that no newform can be related to a non-trivial primitive solution of
(\ref{eq:42p}). By an idea due to Mazur (see \cite{MR3098134}), one
can discard all forms $g$ whose coefficient field $K_g$ does not match
that of $f$ when the prime $p$ is large enough. This justifies the
first hypothesis in Question~1. However, the newforms $g$ whose
coefficient field $K_g$ matches $K_f$ could pass this elimination
procedure. There is a plausible situation that might appear (because
$K$ is an imaginary quadratic field) which is that the building block
of $A_g$ (see \textsection 2.1 for a quick review of building blocks)
might have dimension two (i.e.\ is related to a ``fake elliptic
curve'', namely an abelian surface with quaternionic multiplication).
	
The abelian variety $A_f$ has a $1$-dimensional building block, namely
the elliptic curve $\E$ itself. Suppose once again that the newform
$g$ does not have complex multiplication.

\vspace{5pt}
	
\noindent {\bf Question 2:} Is it true that the building block $E_g$
of $A_g$ has dimension $1$? If so, what is the minimum field of
definition of the elliptic curve $E_g$?
	
\vspace{5pt}
	
To our knowledge no general method was developed before to provide an
answer to Question 2 (i.e.\ an unpleasant type of congruence between
an elliptic curve and a fake elliptic curve). A key result used in the
pioneering article \cite{MR3811755} is that fake elliptic curves have
potentially good reduction at all primes. In particular, if the
elliptic curve attached to a putative solution of our favorite
Diophantine equation has a prime of multiplicative reduction, the
building block $E_g$ must have dimension $1$. This is the case for
Fermat's original equation, as exploited in \cite{MR3811755} while
proving asymptotic results for general number fields. Unfortunately,
this is not the case for many Diophantine equations, like equation
(\ref{eq:42p}), which motivated the results of the present article. A
possible workaround (to get partial results) is to impose some
constraints on the solutions in order to ensure a prime of potentially
multiplicative reduction for the attached elliptic curve. For
instance, this idea was used in \cite{IKOpp2, mocanu} while studying
equations $x^p+y^p=z^2$ and $x^p+y^p=z^3$.
	
One of the main contributions of the present article is to provide a
positive answer to Question 2 when $p$ is large enough (see
Proposition~\ref{prop:decomposition} and
Theorem~\ref{thm:fielddefinition}). Furthermore, we prove that the
elliptic curve $E_g$ can be defined over the quadratic field $K$ (see
Theorem~\ref{thm:fielddefinition}) and the building block is totally
defined over $K(\sqrt{-2})$. A non-trivial strengthening of our
solution to this problem (required while studying
equation~(\ref{eq:42p})) is that the curve $E_g$ can be chosen so that
residual Galois representations $\bar{\rho}{_{\E,p}}$ and
$\bar{\rho}{_{E_g,p}}$ are isomorphic (see
Theorem~\ref{thm:main-decomposition}).
	
The method used to answer Question 2 could be used to prove
non-existence of solutions of other Diophantine equations over number
fields (like imaginary quadratic ones).
	
For proving non-existence of non-trivial primitive solutions
of~(\ref{eq:42p}), the last missing ingredient is a result on
non-existence of elliptic curves with the same properties as $E_g$
defined over $K$. The key property that $E_g$ satisfies is that it has
conductor supported on primes dividing $2$ and it has a $K$-rational
point of order $2$.  Some quite recent results on Diophantine
equations depend on results of non-existence of elliptic curves over
number fields whose conductor is supported at a unique prime (see for
example~\cite{MR4054049}). Here is an instance of such a result that
we prove in the present article.
	
\begin{thmB}
  Let $d\neq 3$ be a prime number such that $d\equiv3\pmod8$, and $3$
  does not divide the class number of $K=\Q(\sqrt{-d})$. Then the only
  elliptic curves defined over $K$ having a $K$-rational point of
  order $2$ and conductor supported at $2$ are those that are base
  change of $\Q$.
\end{thmB}
	
As a consequence we can prove the following asymptotic result.
	
\begin{thmC}
  Let $d$ be a prime number congruent to $3$ modulo $8$ and such that
  the class number of $\Q(\sqrt{-d})$ is not divisible by $3$. Then
  there are no non-trivial primitive solutions of the equation
  \[
    x^4+dy^2=z^p,
  \]
  for $p$ large enough.
\end{thmC}
	
A similar approach works while studying the Diophantine equation
\begin{equation}
  \label{eq:2}
  x^2+dy^6 = z^p.
\end{equation}
To a putative solution $(a,b,c)$ one can attach the elliptic curve
\begin{equation}
  \label{eq:ell-curve-2}
  \Et: y^2+6b\sqrt{-d}xy-4d(a+b^3\sqrt{-d})y = x^3,
\end{equation}
over the quadratic field $K=\Q(\sqrt{-d})$.  The curve $\Et$ is again
a $\Q$-curve with a $K$-rational point of order $3$ (namely the point
$(0,0)$). Such equation was also studied in \cite{PT}. Once again,
there is a character $\varkappa$ such that the twisted representation
$\rho_{\Et,p} \otimes \varkappa$ extends to an odd representation of
$\Gal_\Q$. The main difference with equation~(\ref{eq:42p}) is that
the elliptic curve $\Et$ has bad additive reduction at all primes of
$K$ dividing $d$ (extra care must be taken at primes dividing $6$, see
\cite[Lemmas 2.13, 2.14 and 2.15]{PT}), while the curve $\E$ had only
bad reduction (either additive or multiplicative, depending on $d$
modulo $8$, see \cite[Lemma 2.8]{PT}) at primes dividing $2$. We will
prove the following result.
	
\begin{thmE}
  Let $d$ be a prime number congruent to $19$ modulo $24$ and such
  that the class number of $\Q(\sqrt{-d})$ is not divisible by
  $3$. Then there are no non-trivial primitive solutions of the
  equation
  \[
    x^2+dy^6=z^p,
  \]
  for $p$ large enough.
\end{thmE}

\vspace{5pt}
	
The article is organized as follows: the first section recalls the
definition and main properties of inner twists as developed by Ribet
in \cite{ribet1980twists}. It also contains the proof of
Theorem~\ref{thm:main} providing an answer to Question 1. In
Section~\ref{section:decomposition}, after recalling the basic
definitions of building blocks and fields of definition, we apply the
theory of inner twists to the abelian variety $A_f$ attached to the
(non-quadratic twist of the) elliptic curve $\E$ coming from a
putative solution $(a,b,c)$ of~(\ref{eq:42p}). In particular, we
compute explicitly the group of inner twists of $A_f$ and use this
information to answer Question 2 (and its consequences). The last part
of the article is devoted to prove Theorem~\ref{thm:nonexistence} on
non-existence of elliptic curves over $K$ with a $2$-torsion point and
bad reduction only at the prime $2$. It also contains the proof of
Theorem~\ref{thm:asymptotic} and of Theorem~\ref{thm:asymptotic2}. The
code used to prove Theorem~\ref{thm:asymptotic2} is available at
\url{https://github.com/lucasvillagra/Asymptotic-results}
\subsection*{Acknowledgments} AP was partially supported by the
Portuguese Foundation for Science and Technology (FCT) Individual Call
to Scientific Employment Stimulus
(\url{https://doi.org/10.54499/2020.02775.CEECIND/CP1589/CT0032}) and
by CIDMA under the FCT Multi-Annual Financing Program for R\&D
Units. LVT was supported by a CONICET grant and FGM was supported by
an FCT grant UI/BD/152573/2022. We thank the anonymous referees for
their comments that improved the quality of the present article.

\section{Inner Twists}
	
Let $f \in S_k(\Gamma_0(N),\varepsilon)$ be a modular form and $\chi$
be a Dirichlet character. The \emph{twist} of $f$ by $\chi$ (denoted
$f\otimes\chi$) is the newform attached to the modular form with
Fourier expansion $\sum_{n\ge1} a_n(f)\chi(n)q^n$.
\begin{defi}
  A modular form $f\in S_k(\Gamma_0(N),\varepsilon)$, with $k\ge2$,  has \textit{complex
    multiplication} (CM for short) if there exists a non-trivial
  Dirichlet character $\chi$ such that $f = f \otimes \chi$.
\end{defi}
	
As mentioned in the introduction, due to the applications we have in
mind, during this article we will restrict to modular forms without
complex multiplication.  Recall the following definition of
\cite{ribet1980twists}.
	
\begin{defi}
  Let $f\in S_2(\Gamma_0(N),\varepsilon)$ be a newform without complex
  multiplication, and let $K_f$ denote its coefficient field. The
  \textit{set of inner twists} of $f$ is defined as
  \begin{multline} \label{defi:inner-twist}
    \Gamma_f \vcentcolon  = \{ \gamma \in \Hom_\Q(K_f,\CC)  : \; \exists \ \chi_\gamma \text{ a Dirichlet character} \  \ \text{with} \\
    \gamma(a_p(f)) = \chi_\gamma(p) a_p(f) \ \text{for almost all} \
    p\}.
  \end{multline}
\end{defi}
	
\begin{remark}
  \label{remark:primeslevel}
  Although it is not explicitly stated in Ribet's article, it is the
  case that we can replace the condition ``for almost all $p$'' in
  (\ref{defi:inner-twist}) by ``for all $p$ not dividing the conductor
  of $\chi_\gamma$ nor $N$''.
\end{remark}
	
Here are some facts about $\Gamma_f$ and its elements:
\begin{enumerate}
\item For $\gamma \in \Gamma_f$, $\gamma(K_f) \subset K_f$ (see
  \cite[Proposition 3.2]{ribet1980twists}). In particular, $\Gamma_f$
  is a subset of $\Aut_\Q(K_f)$ and the values of $\chi_\gamma$ belong
  to $K_f$.
\item The set $\Gamma_f$ is in fact an abelian group (see
  \cite[Proposition 3.3]{ribet1980twists}).
		
\item Given $\gamma \in \Gamma_f$, the character $\chi_{\gamma}$ is
  unique (see \cite{ribet1980twists}, page 48).
\item The conductor of $\chi_\gamma$ is supported at primes dividing
  $N$ (see \cite{ribet1980twists}, page 48).
\end{enumerate}
	
The third property implies that we can (and will) denote elements of
$\Gamma_f$ by pairs $(\gamma, \chi)$.

\begin{exm}
  If $f$ is a newform in $S_2(\Gamma_0(N),\varepsilon)$, and the
  Nebentypus $\varepsilon$ is not trivial, then the coefficient field
  $K_f$ is a CM extension of $\Q$, and the pair
  $(c,\varepsilon^{-1})$, where $c$ denotes complex conjugation, is an
  element of $\Gamma_f$ (see \cite[Example 3.7]{ribet1980twists}).
\end{exm}

\begin{lemma}
  \label{lemma:conductor-bound}
  If $(\gamma, \chi) \in \Gamma_f$ then the conductor of $\chi$
  divides $8N$.
\end{lemma}
\begin{proof}
  Comparing the Nebentypus of $\gamma(f)$ and $f \otimes \chi$ we get
  the well known
  relation
  \begin{equation}
    \label{eq:twist-neb}
    \chi^2 = \gamma(\varepsilon)/\varepsilon. 
  \end{equation}
  In particular the conductor of $\chi^2$ divides $N$. If $p$ is an
  odd prime, the valuation at $p$ of the conductor of $\chi^2$ is
  either $0$ (so the valuation at $p$ of the conductor of $\chi$ is at
  most $1$), or it equals the valuation at $p$ of the conductor of
  $\chi$. But the conductor of $\chi$ is supported at primes dividing
  $N$ (because $f \otimes \chi$ has the same level $N$ as $f$) so we
  conclude that the $p$-valuation of its conductor is at most $v_p(N)$. At the
  prime 2, if $\chi^2$ is unramified at $2$, then the conductor of
  $\chi$ has valuation at most $3$ at $2$; otherwise, the valuation at
  $2$ of the conductor of $\chi^2$ is one less than that of $\chi$.
\end{proof}

Let $\End^0(A_f)\vcentcolon =\End(A_f) \otimes \Q$ denote the algebra
of endomorphisms defined over the algebraic closure of $\Q$. If $L$ is
an extension of $\Q$, let $\End^0_L(A_f) := \End_L(A_f) \otimes \Q$,
where $\End_L(A_f)$ denotes the ring of endomorphisms of $A_f$ defined
over $L$.  An important consequence of the main results of
\cite{ribet1980twists} is that the endomorphism algebra $\End^0(A_f)$
can be computed in terms of the group of inner twists $\Gamma_f$.
Concretely, in the proof of \cite[Theorem 5.1]{ribet1980twists} Ribet
constructs for each inner twist $(\gamma,\chi) \in \Gamma_f$ an
endomorphism $\eta_{\gamma}$ and proves that the algebra generated by
$K_f$ and the endomorphisms $\eta_{\gamma}$ (for all
$\gamma \in \Gamma_f$) equals $\End^0(A_f)$ (see page 59 of
\cite{ribet1980twists}).
	
The construction of $\eta_{\gamma}$ given by Ribet is as follows: let
$(\gamma,\chi) \in \Gamma_f$ be an inner twist. Let $\ord$ be the
order of $\chi$ and let $\cond$ be its conductor.
As mentioned in Ribet's article, without loss of generality, we can
assume that $\cond^2 \mid N$. Otherwise, let $h$ be the modular form
\[
  h = \sum_{(n,r)=1} a_n(f) q^n.
\]
Then $h$ is a modular form of level $r^2N$ (as proved for example in
\cite[\S 3, Proposition 17]{MR1216136}). In \cite{ribet1980twists}
Ribet proves (see Proposition 2.2 and page 57) that the abelian
variety $A_f$ is isogenous over $\Q$ to the abelian variety $A_h$,
hence it is enough to define the endomorphism $\eta_\gamma$ on $A_h$
(see \S 5 of \cite{ribet1980twists} for extra details).

The first stated property at the beginning of the section implies that
$K_f$ contains the $\ord$-th roots of unity, so for any integer $u$,
the value $\chi^{-1}(u)$ is an endomorphism of $A_f$ (defined over
$\Q$). Then (as in page 57 of \cite{ribet1980twists}) let
\begin{equation}
  \label{eq:endom}
  \eta_\gamma \vcentcolon  = \sum_{u \pmod \cond} \chi^{-1}(u) \circ \alpha_{u/\cond},  
\end{equation}
where $\alpha_{u/\cond}$ is the endomorphism corresponding to slashing
by the matrix
$\left(\begin{smallmatrix} 1 & \frac{u}{\cond}\\ 0 &
    1\end{smallmatrix}\right)$ in the space $S_2(\Gamma(N))$ (see \S 2
of \cite{ribet1980twists} and also \S 4 of \cite{MR318162}). The
hypothesis $\cond^2 \mid N$ is needed for the slashing operator to
preserve the space.
	
As explained in Ribet's article, the endomorphism $\eta_\gamma$ is
defined over the field of $\cond$-th roots of unity (because the map
$\alpha_{u/\cond}$ is defined over such a field (as explained in \S 4
of \cite{MR318162}, see also \S 6 of \cite{MR1291394}). The following
result is well known to experts, but we did not find a proper
reference for its proof. Let $\Q(\zeta_\cond)$ be the field of
$\cond$-th roots of unity. Identify $\chi$ with a character of
$\Gal(\Q(\zeta_\cond)/\Q)$ and let $\overline{\Q}^\chi$ be the field
fixed by the kernel of $\chi$.
\begin{lemma}
  The endomorphism $\eta_\gamma$ is defined over $\overline{\Q}^\chi$.
  \label{lemma:fielddefinition}
\end{lemma}
\begin{proof}
  We know that $\eta_\gamma$ is defined at least over
  $\Q(\zeta_\cond)$. Let $\sigma \in \Gal(\Q(\zeta_\cond)/\Q)$, say
  $\sigma(\zeta_\cond) = \zeta_\cond^i$. If $f \in S_2(\Gamma_1(N))$
  then $\sigma(\alpha_{u/r}(f)) = \alpha_{iu/r}(\sigma(f))$ (by
  looking at the $q$-expansion). This relation implies the relation
  $\sigma(\alpha_{u/\cond}) = \alpha_{iu/\cond}$ as endomorphisms of
  $A_f$. Since the endomorphism $\chi(u)$ is defined over $\QQ$,
  \[
    \sigma(\eta_\gamma) = \sum_{u\pmod \cond}\chi^{-1}(u) \circ
    \sigma(\alpha_{u/\cond}) = \chi(i) \sum_{v \pmod \cond}
    \chi^{-1}(v)\circ \alpha_{{v/\cond}},
  \]
  where the second equality comes from the change of variables $v=iu$.
  The result follows from the fact that $\chi(i)=1$ if and only if
  $\sigma$ restricted to $\overline{\Q}^\chi$ is the identity.
\end{proof}
	
It was already known to Hecke (see~\cite{MR0371577}, p. 811, Satz 1
and Satz 2) that two modular forms of weight $k$ whose first
coefficients coincide (up to an explicit bound $C$ depending on the
level and the weight of the two forms) must be equal. For later
purposes, we need a little variant of Hecke's result, whose elegant
proof was communicated to us by Professor Gabor Wiese. 
	
\begin{lemma}
  \label{lemma:hecke}
  Let $f\in S_{k}(\Gamma_1(N))$ and $g\in S_{k}(\Gamma_1(\tilde{N}))$
  be newforms. Let $S$ be a finite set of primes. There exists a
  constant $C$ depending only on $N, \tilde{N}, k$ and $S$, such that if
  $a_q(f)=a_q(g)$ for all primes $q \le C$, $q \not \in S$, then
  $f=g$.
\end{lemma}
	
\begin{proof}
  Let $\chi$ be a character whose conductor is divisible by all primes
  in $S$. Applying Hecke's result to $f \otimes \chi$ and
  $g \otimes \chi$, there exists a constant $C$ (depending only on
  $k, N,\tilde{N}$ and the primes in $S$) such that if
  $a_n(f) \chi(n) = a_n(g)\chi(n)$ for all $n \le C$, then
  $f \otimes \chi = g \otimes \chi$. We claim that the constant $C$
  suffices for our purposes.
		
  If $a_q(f) = a_q(g)$ for all primes $q \le C$, $q \not \in S$, then
  (since $f$ and $g$ are newforms) $a_n(f) = a_n(g)$ for all positive
  integers $n \le C$ not divisible by primes in $S$, so
  $a_n(f)\chi(n) = a_n(g) \chi(n)$ for all $n \le C$ and
  $f \otimes \chi = g \otimes \chi$. Since $f$ (respectively $g$) is a
  newform,
  $f = (f\otimes \chi)\otimes\chi^{-1}=(g \otimes \chi)\otimes
  \chi^{-1}=g$.
\end{proof}
	
The following result provides a partial answer to Question 1.
	
\begin{thm}
  \label{thm:main}
  Let $g \in S_2(\Gamma_0(\tilde{N}),\varepsilon)$ be a newform
  without complex multiplication.  There exists a constant $M_g$
  (depending only on $g$) such that if
  $f \in S_2(\Gamma_0(N),\varepsilon)$ is any newform without complex
  multiplication, satisfying the conditions:
  \begin{enumerate}
  \item The coefficient field $K_f$ of $f$ equals the coefficient
    field $K_g$ of $g$.
			
  \item $\tilde{N} \mid N$ and $N/\tilde{N}$ is square-free.
			
  \item There exists a prime $p>M_g$ and a prime ideal $\id{p}$ of
    $K_g$ dividing $p$ such that the Galois representations
    $\bar{\rho}{_{f,\id{p}}}$ and $\bar{\rho}{_{g,\id{p}}}$ are
    isomorphic.
  \end{enumerate}
  Then for any field extension $L/\Q$ there exists an injective
  morphism between the $\Q$-algebras
  \[
    \psi: \End_L^0(A_f) \to \End_L^0(A_g).
  \]
  Moreover, if $L/\Q$ is Galois, then the morphism $\psi$ is also a
  morphism of $\Gal(L/\Q)$-modules.
\end{thm}

\begin{proof}
  The endomorphism algebra $\End^0(A_f)$ is generated by $K_f$ and by
  the endomorphisms $\eta_\gamma$ for $(\gamma,\chi) \in \Gamma_f$
  (the latter defined over the field $\overline{\Q}^\chi$ by
  Lemma~\ref{lemma:fielddefinition}). If we prove that the set
  $\Gamma_f$ is contained in the set $\Gamma_g$, then the morphism
  $\psi$ we seek for sends the endomorphism $\eta_{\gamma}$ of $A_f$
  to the endomorphism $\eta_{\gamma}$ of $A_g$; this morphism is
  clearly injective and Galois equivariant. The key point to prove the
  inclusion of the inner twists groups is the fact that the group
  $\Gamma_f$ is defined in terms of a property of Fourier
  coefficients.
		
  Let $S$ be the set of primes dividing $2\tilde{N}$ and let $M$ be
  the set of characters $\theta$ whose conductor is supported at
  primes dividing $\tilde{N}$ with the property that there exist
  $\gamma \in \Gamma_f$ satisfying (\ref{eq:twist-neb}).
  Lemma~\ref{lemma:conductor-bound} implies that the set $M$ is
  finite. If $\theta \in M$  and
  $\gamma \in \Hom_\Q(K_g,\CC)$ Lemma~\ref{lemma:hecke} implies
  the existence of a constant $C_\theta$ (depending only on
  $\tilde{N}$, $\theta$ and $S$) such that if
  \begin{equation}
    \label{eq:innertwistg}
    a_q(\gamma(g))=\theta(q)a_q(g), \quad \forall q\le C_\theta, \; q\notin S,
  \end{equation}	
  then $\gamma(g)=g \otimes \theta$. Define the constant
  \[
    M_g:=\max_{\theta \in M}\{C_{\theta}^2\}.
  \]
		
  Let $f$ be a newform satisfying the stated hypotheses, so in
  particular there exists a prime $p>M_g$ such that the residual
  Galois representations $\bar{\rho}{_{f,\id{p}}}$ and
  $\bar{\rho}{_{g,\id{p}}}$ are isomorphic for some prime ideal
  $\id{p}$ of $K_g$ dividing $p$.  Let $(\gamma,\chi)\in \Gamma_f$, so
  for all primes $q$ not dividing $2N$,
  \begin{equation}
    \label{eq:1}
    \gamma(a_q(f)) = \chi(q) a_q(f).    
  \end{equation}
  Let $p$ be a prime dividing $N$ but not dividing $\tilde{N}$. Since
  $\varepsilon$ (the Nebentypus of $f$) is unramified at $p$,
  (\ref{eq:twist-neb}) implies that either $\chi$ is unramified at
  $p$, or it is ramified at $p$ but its square is not. We claim that
  the second hypothesis implies that the latter case cannot occur.

  The eigenform $f$ has associated an automorphic representation
  $\pi_f$ of the group $\GL_2(\AA_\Q)$, where $\AA_\Q$ denotes the
  ad\`ele ring of $\QQ$. The representation $\pi_f$ factors as a
  restricted tensor product of components $\pi_{f,v}$ over places $v$
  of $\Q$.  The second hypothesis implies that if $p \mid N$ and
  $p \nmid \tilde{N}$, then the local component $\pi_{f,p}$ is a
  Steinberg representation. Then the level of the twisted modular form
  $f \otimes\chi$ has valuation $1$ at $p$ if $\chi$ is unramified at
  $p$ or $2$ if $\chi$ is ramified at $p$ when $p$ is odd. When $p=2$
  and $\chi$ is ramified at $2$, the valuation at $2$ of the level of
  $f \otimes \chi$ equals $2$ or $6$. Since $\gamma(f)$ has the same
  level as $f$, $\chi$ must be unramified at $p$, proving the
  claim. Then $\chi \in M$, and it is enough to prove the equality
  \begin{equation}
    \label{eq:2}
    \gamma(a_q(g)) = \chi(q) a_q(g),\quad \forall q \le \sqrt{M_g}, q \not \in S,    
  \end{equation}
  to deduce that $(\gamma,\chi) \in \Gamma_g$.  Let $q\le \sqrt{M_g}$
  be a prime number not in $S$. There are two possibilities: either
  $q$ divides $N$ or it does not. If $q\nmid N$, Remark
  \ref{remark:primeslevel} implies that
  \[
    \gamma(a_q(f)) = \chi(q) a_q(f).
  \]
  The third hypothesis implies that
  \begin{equation}\label{eq:cong}
    a_q(g)\chi(q)\equiv a_q(f)\chi(q)\pmod{\id{p}} \ \text{ and } \ \gamma(a_q(g))\equiv \gamma(a_q(f))\pmod{\gamma(\id{p})}.
  \end{equation}
  By the Ramanujan-Petersson conjecture (proved in \cite{MR340258}),
  for all embeddings $\sigma:K_g \to \CC$,
  $|\sigma(a_q(g))| \le 2\sqrt{q}$. Then, since
  \[
    \text{Norm}(\id{p}), \text{Norm}(\gamma(\id{p}))\ge p>M_g\ge
    4\sqrt{q},
  \]
  both sides of each congruence of~(\ref{eq:cong}) are in fact equal,
  so equation (\ref{eq:2}) follows from (\ref{eq:1}).
		
  Suppose then that $q\mid N$ but $q \nmid 2\tilde{N}$ (since
  $q \not \in S$). Then hypotheses (2) and (3) imply that $q$ is a
  prime of ``level lowering'' for $f$ modulo $\id{p}$ so
  \[
    a_q(g)^2\equiv \varepsilon(q)(q+1)^2 \pmod{\id{p}}.
  \]
  Since the absolute value of the left hand side is bounded by
  $4q$, both sides of the congruence are not equal. Then $p$
  must divide their difference, which is bounded by $(\sqrt{q}+1)^2$,
  giving a contradiction, since
  $\text{Norm}(\id{p}) \ge p>M_g\ge q^2> (\sqrt{q}+1)^2$ (since $q$ is
  odd). Then this last case cannot happen.
\end{proof}
	
\section{Applications to the equations $x^4+dy^2=z^p$ and
  $x^2+dy^6=z^p$}
\subsection{Decomposing the abelian variety attached to $\E$}
\label{section:decomposition}
Let us start this section recalling some general basic definitions
(see for example the second chapter of~\cite{MR2693262}). Let $A$ be
an abelian variety of $\GL_2$-type and let $B$ be a simple component
(over $\overline{\Q}$) of $A$.
\begin{defi}
  The simple abelian variety $B$ is a \emph{building block} of $A$ if
  it satisfies:
  \begin{itemize}
  \item The variety $B$ is a $\Q$-variety, i.e.\ it is isogenous to
    all of its Galois conjugates,
			
  \item The endomorphism algebra $\End^0(B)$ is either a totally real
    field $F$ of degree $[F:\Q] = \dim B$ or a totally indefinite
    quaternion algebra over a totally real field $F$ of degree
    $[F:\Q]=\frac{1}{2}\dim B$.
  \end{itemize}
\end{defi}
	
\begin{defi}
  Let $L$ be a number field. A building block $B$ of the variety $A$
  is \emph{totally defined} over $L$ if the abelian variety $B$ is
  defined over $L$, all the isogenies between $B$ and its Galois
  conjugates are defined over $L$ and all of its endomorphisms are
  defined over $L$ as well.
\end{defi}
	
Let $(a,b,c)$ be a non-trivial primitive solution of~(\ref{eq:42p})
for $d \neq 1$ and let $\E$ be the elliptic curve over
$K=\Q(\sqrt{-d})$ defined in (\ref{eq:curve-E}).
	
As explained in the introduction, there exists a finite order Hecke
character $\varkappa$ (whose construction is given in \S 3 of
\cite{PT}) of $K$ unramified outside $2$ and primes ramifying in
$K/\Q$ such that the twisted representation
$\rho_{\E,p}\otimes \varkappa$ extends to a representation
$\tilde{\rho}:\Gal_\Q \to \GL_2(\overline{\Q_p})$. Let
$f \in S_2(\Gamma_0(N),\varepsilon)$ be the newform attached to
$\tilde{\rho}$ (see \cite{PT} for a description of the Nebentypus
$\varepsilon$) and $A_f$ be the $\GL_2$-type abelian variety
constructed via the Eichler-Shimura map.
	
Over $\overline{\Q}$ the variety $A_f$ is isogenous to a product of
simple abelian varieties, $A_f \sim B_1 \times \cdots \times B_k$,
each variety $B_i$ being a building block of $A_f$ as defined
before. In the particular case of abelian varieties coming from
newforms, all building blocks are isogenous to each other, so in
particular $A_f \sim B^k$ (see \cite{MR2448720}).
	
\begin{lemma}
  \label{lemma:no-CM}
  The curve $\E$ does not have complex multiplication if $p>2$.
\end{lemma}
	
\begin{proof}
  Since $K$ is an imaginary quadratic field, if $\E$ has complex
  multiplication, then its $j$-invariant must be a rational number (in
  particular, a real one).
		
  The $j$-invariant of the elliptic curve $\E$ equals
  $j=\frac{64(5a^2-3b\sqrt{-d})^3}{c^p(a^2+b\sqrt{-d})}$.  Since
  $(a,b,c)$ is a non-trivial solution, $a$ and $b$ are non-zero, so
  $j$ is a real number if and only if
  \begin{equation*}
    \begin{cases}
      jc^p=8000a^4-8640db^2\\
      jc^p=-14400a^4+1728db^2.
    \end{cases}
  \end{equation*}
  Subtracting both equations gives the relation
  \[
    175a^4 = 81db^2,
  \]
  hence $\frac{a^2}{b} = \pm \frac{9}{5} \sqrt{d/7}$. Since $d$ is
  square-free, and both $a,b$ are integers, $d=7$ and
  $(a,b)=(\pm3, \pm5)$. Since $c^p = a^4+db^2=256 = 2^8$ we get that
  $p=2$ and $c =\pm 16$.
\end{proof}
	
\begin{lemma}
  Let $r = [K_f:\Q]$. Then the endomorphism algebra $\End^0(A_f)$ is
  isomorphic to $M_r(\Q)$.
\end{lemma}
\begin{proof}
  Let $L=\overline{K}^{\varkappa}$. Then there are isomorphisms of
  Galois representations
  $\rho_{A_f,p}|_{\Gal_L} \simeq (\rho_{f,p}|_{\Gal_L})^r \simeq
  (\rho_{\E,p}|_{\Gal_L})^r$, so Faltings' isogeny theorem (see
  Corollary 1 in page 21 of \cite{MR861971}) implies that $A_f$ is
  isogenous to $(\E)^r$, hence the elliptic curve $\E$ is a building
  block of $A_f$ which does not have complex multiplication.
\end{proof}
	
\begin{remark}
  The elliptic curve $\E$ is a building block of $A_f$ defined over
  $K$, but it is totally defined over $K(\sqrt{-2})$.
  \label{rem:ell-curve}
\end{remark}
	
\begin{remark}
  Since the center of $\End^0(A_f)$ is the field of rational numbers
  $\Q$, Ribet's result (Theorem 5.1 and the remark before Proposition
  3.5 of \cite{ribet1980twists}) implies that the field generated by
  the numbers $a_p(f)^2 \varepsilon(p)^{-1}$ for $p$ not dividing the
  level of $f$ is the rational one. Let us just verify that this is
  indeed the case (because a similar computation will be needed
  later).
		
  If $K$ is a number field, we denote by $\II_K$ its id\`ele ring. A
  key property of the characters $\varkappa$ and $\varepsilon$ is that
  as characters of the respective id\`ele group, they satisfy the
  relation
  \begin{equation}
    \label{eq:relation}
    \varkappa^2 = \varepsilon \circ \norm,  
  \end{equation}
  where $\norm:\II_K \to \II_\Q$ is the norm map (see \cite{PT}, page
  2831). Consider the following two cases:
  \begin{itemize}
  \item If the prime $p$ splits, say $p = \frak{p}_1 \frak{p}_2$, then
    relation~(\ref{eq:relation}) translates into
    $\varkappa (\frak{p}_1) ^2 = \varepsilon(p) $. Then
    \[
      a_p(f)^2 \varepsilon(p)^{-1} = \left(a_{\frak{p}_1}(E)
        \varkappa(\frak{p}_1)\right)^2 \varepsilon(p)^{-1}=
      a_{\frak{p}_1}(E)^2 \in \QQ.
    \]
  \item If the prime $p$ is inert, relation~(\ref{eq:relation})
    implies that
    $\varkappa (p) ^2 = \varepsilon(p^2) = \varepsilon(p)^2 $. Then
    using the relation between $a_p(f)$ and $a_p(E)$,
    \[
      a_p(f)^2 \varepsilon(p)^{-1} =
      a_p(E)\varkappa(p)\varepsilon(p)^{-1}+2p = \pm a_p(E) -2p \in
      \QQ.
    \]
  \end{itemize}
\end{remark}
	
\vspace{5pt}
	
The elliptic curve $\E$ has discriminant
$\Disc(\E) = 512(a^2+b\sqrt{-d})c^p$. From the equality
\[
  (a^2+b\sqrt{-d})(a^2-b\sqrt{-d})=a^4+db^2=c^p
\]
and the hypothesis that $(a,b,c)$ is a primitive solution (so any
prime ideal dividing both $a^2+b\sqrt{-d}$ and $a^2-b\sqrt{-d})$ must
divide $2$) it follows that $(a^2+b\sqrt{-d})$ is a $p$-th power
outside $2$, and the same holds for $\Disc(\E)$. Let $q \nmid 2d$ be a
prime number. If $q \mid c$ and $q \neq p$, then the curve $\E$ has
multiplicative reduction at the primes dividing $q$ (in $K$), and the
residual representation $\bar{\rho}{_{\E,p}}$ is unramified at (the
primes dividing) $q$. Since $\varkappa$ is unramified at $q$ (by
construction), the same holds for
$\overline{\rho_{\E,p} \otimes \varkappa}$.  If $q = p$, then the
residual representation corresponds to a finite flat group scheme (or
equivalently, its Serre's weight equals $2$) by a result due to
Hellegouarch, and since $\varkappa$ is unramified at (primes dividing)
$p$, the same holds for the twisted representation. Since the
extension $K/\Q$ is unramified at primes not dividing $2d$, the
residual representation of $\tilde{\rho}=\rho_f$ is also unramified at
all primes $q$ not dividing $2d$.
	
There exists an explicit bound $N_K$ such that the residual Galois
representation $\bar{\rho}{_{\E,p}}$ is absolutely irreducible for all
primes $p > N_K$ (see \cite[Theorem 5.1] {PT} and \cite[Proposition
3.2]{MR2075481}).  Then by Ribet's lowering the level result applied
to $f$, there exists a newform
$g \in S_2(\Gamma_0(\tilde{N}), \varepsilon)$, where $\tilde{N}$ is a
positive integer only divisible by $2$ and by the primes dividing $d$,
such that $\bar{\rho}{_{f,\id{p}}} \simeq \bar{\rho}{_{g,\id{P}}}$,
where $\id{p}$ and $\id{P}$ are some primes in $K_f$ and $K_g$
(respectively) dividing $p$. As explained in the introduction, it will
be sufficient to consider the case when $K_f=K_g$ and
$\id{p}=\id{P}$. Note in particular that the value of $\tilde{N}$ is
independent of the solution $(a,b,c)$ we started with.
	
Suppose that the form $g$ does not have complex multiplication.  Let
$A_g$ be the abelian variety attached to the newform $g$ by
Eichler-Shimura's construction. An immediate consequence of
Theorem~\ref{thm:main} is the following result.
	
\begin{prop}
  \label{prop:decomposition}
  Suppose that $K_f = K_g$. There exists a constant $B$ (depending
  only on $\tilde{N}$) such that if $p>B$, then we have
  $\End^0(A_g) \simeq \End^0(A_f)$ $\simeq M_r(\Q)$. In particular,
  the building block of the abelian variety $A_g$ has dimension $1$.
\end{prop}
\begin{proof}
  From the theory of building blocks, we know that there exists a
  simple abelian variety $E$ such that $A_g \sim E^t$ hence
  $\End^0(A_g)\simeq M_t(\End^0(E))$.  Theorem~\ref{thm:main} implies
  that $\End^0(A_f) \subset \End^0(A_g)$ (over $\overline{\Q}$ under
  the map $\psi$), so $M_r(\Q) \subset M_t(\End^0(E))$, where
  $r = [K_f:\Q] = \dim(A_g)=t\dim(E)$, implying that $r \le t$ (since
  $\End^0(E)$ does not have zero divisors). Then $r = t$ and
  $\dim(E)=1$, i.e.\ $E$ is an elliptic curve. Since the form $g$ does
  not have complex multiplication, neither does $E$ hence
  $M_r(\Q)=\End^0(A_g)$.
\end{proof}
	
It is a natural problem to determine the minimal field of definition
(if it exists) of a building block of $A_g$ and whether it matches a
building block of $A_f$ (namely $K$).
	
\begin{thm}
  There exists a $1$-dimensional building block $E_g$ for $A_g$
  defined over the quadratic field $K$ and totally defined over
  $K(\sqrt{-2})$.
  \label{thm:fielddefinition}
\end{thm}
	
\begin{proof}
  Let $E_g$ denote any building block of $A_g$ (which is
  $1$-dimensional by the last proposition).  Recall that $E_g$ is a
  $\Q$-curve, i.e.\ the curve $E_g$ is totally defined over a Galois
  number field $L$ satisfying that for all $\sigma \in \Gal(L/\Q)$,
  the curve $\sigma(E_g)$ is isogenous to $E_g$. Let
  $\mu_\sigma \colon \sigma(E_g) \to E_g$ denote such an
  isogeny. Abusing notation (as in Ribet's article
  \cite{ribet1980twists}) we can attach to $E_g$ a map
  $c : \Gal_\Q \times \Gal_\Q \to \Q^\times$ given by
  \begin{equation}
    \label{eq:cocycle}
    c(\sigma,\tau) = \mu_\sigma \circ \sigma(\mu_\tau)\circ \mu_{\sigma \tau}^{-1},
  \end{equation}
  which is an element of $\End^0(E_g) \simeq \Q^\times$. The map $c$
  is actually a cocycle (by \cite{ribet1980twists} (5.7)). In
  particular, its class is an element of $\HH^2(\Gal_\Q,\Q^\times)$,
  whose order is at most $2$ (see \cite[Remark 5.8]{ribet1980twists}
  and \cite[Proposition 3.2]{MR1299737}). Then by \cite[Proposition
  5.2]{MR2693262} the building block $E_g$ is isogenous (over
  $\overline{\Q}$) to a building block totally defined over a field
  $F$ if and only if $[c]$ lies in the kernel of the restriction map
  $\text{Res:}\HH^2(\Gal_\Q,\Q^\times) \to \HH^2(\Gal_F,\Q^\times)$.
		
  By a result of Ribet (\cite[Corollary 4.5]{MR1299737}) the curve
  $E_g$ does have a minimum field of totally definition. Furthermore,
  it can be explicitly described (as done in the proof of
  \cite[Theorem 3.3]{MR1299737}): consider the natural isomorphism
  \[
    \Q^\times \simeq \{\pm 1\} \times \Q^\times/\{\pm 1\},
  \]
  where now the second factor is a free group that can be identified
  with the group of positive rational numbers $\Q^\times_+$. This
  induces an isomorphism
  \[
    \HH^2(\Gal_\Q,\Q^\times)[2]\simeq \HH^2(\Gal_\Q,\Q^\times_+)[2]
    \times \HH^2(\Gal_\Q,\{\pm 1\}).
  \]
  The short exact sequence
  \[
    \xymatrix{ 1 \ar[r]& \Q^\times_+ \ar[r]^{x \mapsto x^2}&
      \Q^\times_+ \ar[r]& \Q^\times_+/(\Q^\times_+)^2 \ar[r]& 1, }
  \]
  induces an isomorphism of the cohomology groups
  $\HH^2(\Gal_\Q,\Q^\times_+)[2] \simeq
  \Hom(\Gal_\Q,\Q^\times_+/(\Q^\times_+)^2)$.  Our cocycle class $[c]$
  then decomposes as a product (following Ribet's notation)
  $(\overline{c},c_{\pm})$, where
  $\overline{c} \in \HH^2(\Gal_\Q,\Q^\times_+)[2]$ and
  $c_{\pm} \in \HH^2(\Gal_\Q,\{\pm 1\})$.  The minimum field of
  totally definition $K_{\text{min}}$ for a building block equals the
  fixed field of $\overline{c}$.
		
  There is a second way to define the cocycle $[c]$ in terms of the
  $\Q$-algebra $\End^0(A_g)$ (see \cite[Chapter 1]{MR2693262}). Let
  $K_g$, as before, denote the coefficient field of $g$ (which also
  equals $\End_\Q^0(A_g)$) and let $\psi \in \End^0(A_g)$. The group
  $\Gal_\Q$ acts on the set $\End^0(A_g)$. Let us denote by
  $^\sigma \psi$ the action of $\sigma \in \Gal_\Q$ on an endomorphism
  $\psi$. Skolem-Noether's theorem implies the existence of an element
  $\alpha(\sigma) \in K_g^\times$ such that
  $^\sigma \psi = \alpha(\sigma) \circ \psi \circ \alpha(\sigma)^{-1}$
  for every $\psi \in \End^0(A_g)$. We can then define a second
  cocycle
  \[
    c (\sigma,\tau) =
    \alpha(\sigma)\alpha(\tau)\alpha(\sigma\tau)^{-1}.
  \]
  Then by \cite[Theorem 4.6]{MR2693262}, both definitions
  coincide. But by Proposition~\ref{prop:decomposition}, the
  $\Q$-algebras $\End^0(A_f)$ and $\End^0(A_g)$ are isomorphic as
  $\Gal_\Q$-modules, hence with this second definition it is clear
  that the cocycle attached to $A_f$ matches the one attached to
  $A_g$, and in particular the minimum field of totally definition of
  both building blocks coincide, so the building block $E_g$ of $A_g$
  can be totally defined over the field
  $\Q(\sqrt{-d},\sqrt{-2})=K(\sqrt{-2})$ (see
  Remark~\ref{rem:ell-curve}). We need to prove that the elliptic
  curve $E_g$ can furthermore be defined over $K$.
		
  Let $\sigma_2 \in \Gal(\Q(\sqrt{-d},\sqrt{-2})/\Q)$ be the map given
  by $\sigma_2(\sqrt{-d})=-\sqrt{-d}$ and
  $\sigma_2(\sqrt{-2})=\sqrt{-2}$. Denote by $\identity$ the identity
  element in such a Galois group. Since the elliptic curve $\E$ is
  defined over $K$, we can take the isogeny $\mu_{\sigma_2}=1$ (the
  identity) and $\mu_{\identity} =1$ (i.e.\ the isogeny corresponding
  to the identity element to be the identity map on $\E$), the
  corresponding map $\widetilde{\mu_{\sigma_2}}$ on $E_g$ satisfies
  that
  \[
    1 = c({\sigma_2},{\sigma_2}) = \widetilde{\mu_{\sigma_2}} \circ
    \sigma_2(\widetilde{\mu_{\sigma_2}}).
  \]
  In particular, $\widetilde{\mu_{\sigma_2}}$ has degree one, hence is
  an isomorphism (so $E_g$ is isomorphic to $\sigma_2(E_g)$). Then the
  $j$-invariant of $E_g$ lies in $K$, so we can take the curve $E_g$
  defined over such a field.
\end{proof}
	
\begin{remark}
  \label{remark:building-block}
  Even when the building block $E_g$ of $A_g$ is defined over $K$, it
  is not true (in general) that if $L/K$ is a field extension where
  the abelian variety $A_g$ has a $1$-dimensional building block $E$,
  then $E$ is isogenous to $E_g$ (this will become clear while proving
  Theorem~\ref{thm:main-decomposition}). Here is an example (that will
  be explained in detail while proving such theorem): let
  $L=K \cdot \overline{\Q}^\varepsilon$, where
  $\overline{\Q}^\varepsilon$ denotes the field fixed by the kernel of
  $\varepsilon$. Then there exists an elliptic curve $E$ defined over
  $L$ (actually $E = \E \otimes \varkappa$) such that the abelian
  variety $A_f$ is isogenous over $L$ to $E^r$ (where
  $r = \dim(A_f)$). However, the building block $E$ is not defined
  over $K$ and is not isomorphic (nor isogenous) to $E_g$ over $L$
  (although they clearly are isomorphic over $\overline{\Q}$).
		
  A natural question (whose answer will be needed later) is the
  following:
		
  \vspace{2pt}
		
  \noindent {\bf Question 3:} Suppose that $L$ is a number field, and
  suppose that an abelian variety $A$ is isogenous over $L$ to $E^r$
  for some building block $E$. When is $E$ the base change of a
  variety (up to isogeny) that is defined over a smaller field $K$?
		
  \vspace{2pt}
		
  Note that if $E$ is defined over $K$, then the cocycle $c$ attached
  to $E$ in the proof of the last theorem is trivial while restricted
  to $\Gal_K$ (not just cohomologically trivial). In \cite[Theorem
  8.2]{MR2058653}, Ribet proves that the converse is also true, i.e.\
  he proves that if the cocycle $c$ is trivial on $\Gal_K$ then there
  exists an abelian variety $\tilde{E}$ defined over $K$ such that $E$
  is isogenous to $\tilde{E}$ over $L$.
\end{remark}
	
To relate the residual Galois representation of $E_g$ to that of our
original elliptic curve $\E$, we need some understanding on the
coefficient field $K_f$. Let $\ordkappa$ be the order of the character
$\varkappa$ and let $\zeta_\ordkappa$ be a primitive $\ordkappa$-th
root of unity. Let $\QQ(\varkappa)=\Q(\zeta_\ordkappa)$ denote the
field obtained by adding to $\QQ$ the values of $\varkappa$.
\begin{lemma}
  Following the previous notation we have that
  $\Q(\zeta_M) \subset K_f$.
  \label{lemma:coef-units}
\end{lemma}
	
\begin{proof}
  The set of prime ideals $\id{p}$ of $K$ which are unramified in
  $K/\Q$ and with inertial degree $1$ over $\Q$ (i.e.\
  $f(\id{p}|p)=1$) have density one in $K$, so by Chebotarev's density
  theorem, there exists a set $S$ of primes with inertial degree $1$
  of positive density (in the set of all primes of $K$) such that
  $\varkappa(\id{p})$ is a primitive $\ordkappa$-th root of unity for
  all prime ideals $\id{p}\in S$.  Our assumption that the curve $\E$
  does not have complex multiplication implies that for some prime
  $\id{p} \in S$, the value $a_{\id{p}}(\E) \neq 0$ (as the set of
  primes $\id{p}$ of good reduction where $a_{\id{p}}(\E)=0$ has
  density zero by \cite{MR1484415}, page IV-13). In particular, for
  such a prime (of norm $p$), it holds that
  \[
    a_p(f) = \varkappa(\id{p}) a_{\id{p}}(\E),
  \]
  is non-zero. The result follows from the fact that
  $\varkappa(\id{p})$ is a primitive $\ordkappa$-th root of unity.
\end{proof}
	
Let $\overline{K}^{\varkappa}$ denote the abelian extension of $K$
fixed by the kernel of the character
$\varkappa: \Gal_K \to \CC^\times$ and similarly let
$\overline{\Q}^\varepsilon$ be the field fixed by the kernel of the
character $\varepsilon$.
	
\begin{lemma}
  With the previous notation,
  $K\cdot\overline{\Q}^\varepsilon \subset \overline{K}^\varkappa$
  with index $2$. Moreover, we have the following field diagram.
  
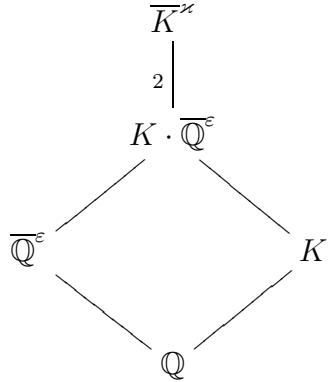
\[
    \xymatrix{
      & \overline{K}^{\varkappa}\ar@{-}[d]_2\\
      & K \cdot \overline{\Q}^\varepsilon\ar@{-}[dl] \ar@{-}[dr]\\
      \overline{\Q}^\varepsilon\ar@{-}[dr] &  & K\ar@{-}[dl]\\
      & \Q }
  \]
  \label{lemma:field-diagram}
\end{lemma}
\begin{proof}
  Follows from the fact that as a Galois character,
  $\varepsilon|_{\Gal_K} = \varkappa^2$.
\end{proof}

\begin{prop}
  The coefficient field $K_f$ is either:
  \begin{enumerate}
  \item The field $\Q(\zeta_\ordkappa)$, or
			
  \item A quadratic extension of $\Q(\zeta_\ordkappa)$.
  \end{enumerate}
  \label{prop:coef-field}
  Moreover, $K_f = \Q(\zeta_\ordkappa,a_p(f))$, where $p$ is any prime
  inert in $K/\Q$ that does not divide the level of $f$, of ordinary
  reduction for $\E$ and such that $a_p(f)\neq0$.
\end{prop}
\begin{proof}
  By Lemma~\ref{lemma:coef-units} we have
  $\Q(\zeta_\ordkappa) \subseteq K_f$. Let $p$ be a rational prime not
  dividing the level of $f$ which is split in $K$, say
  $p = \id{p}\overline{\id{p}}$. Then
  \begin{equation}
    a_p(f) = a_{\id{p}}(\E) \varkappa(\id{p}) \in \Q(\zeta_\ordkappa).
  \end{equation}
  On the other hand, if $p$ is an inert prime, we have the formula
  \begin{equation}
    \label{eq:inert}
    a_p(f)^2 = a_{p}(\E)\varkappa(p) +2p\varepsilon(p).
  \end{equation}
  Recall that $\varkappa^2 = \varepsilon \circ \norm$, so
  $\varkappa(p) = \pm \varepsilon(p)$. Thus
  $a_p(f)^2 \in \Q(\zeta_M)$. Formula~(\ref{eq:inert}) also implies
  that for a fixed inert prime $p$, the extension
  $L=\Q(\zeta_\ordkappa)(a_p(f))$ has degree at most two over
  $\Q(\zeta_\ordkappa)$, and clearly $L \subset K_f$.
		
  Let $\ell$ be a rational prime, and let $\lambda$ be a prime in $L$
  dividing it and let $\id{l}= \lambda \cap \Q(\zeta_\ordkappa)$. In
  the usual basis, the twisted representation
  $\rho_{\E,\ell} \otimes \varkappa$ takes values in
  $\GL_2(\Q(\zeta_\ordkappa)_{\id{l}})$. To extend it to a
  representation $\tilde{\rho}$ of $\Gal_\Q$ it is enough to define it
  on an element $\sigma \in \Gal_\Q$ which is not in $\Gal_K$, for
  example a Frobenius element $\Frob_p$ at a prime $p$ inert in $K$.
		
  To ease notation, let $t = a_p(f) = \trace(\tilde{\rho}(\Frob_p))$
  and $s = p \varepsilon(p) = \det(\tilde{\rho}(\Frob_p))$.  Assume
  that $a_p(f)=t \neq 0$ and that $p$ is a prime of ordinary reduction
  for $\E$. The ordinary hypothesis on $p$ implies that
  $t \neq 2\sqrt{s}$ (otherwise (\ref{eq:inert}) gives that
  $a_p(f) = \pm 2p$ so $p$ is not ordinary for $\E$). The matrices
  $\rho_{\E,\ell}(\Frob_p^2)\varkappa(\Frob_p^2)$ and
  $\left(\begin{smallmatrix}-s & -st \\ t &
      t^2-s\end{smallmatrix}\right)$ are diagonalizable, have the same
  trace and the same determinant, hence there exists a matrix
  $W \in \GL_2(L_\lambda)$ such that
  \[
    W (\rho_{\E,\ell}(\Frob_p^2)\varkappa(\Frob_p^2)) W^{-1} =
    \left(\begin{smallmatrix}-s & -st \\ t &
        t^2-s\end{smallmatrix}\right).
  \]
  Conjugating the representation $\rho_{\E,\ell} \otimes \varkappa$ by
  $W$, we can assume, without loss of generality, that this twisted
  representation takes values in $\GL_2(L_\lambda)$ and that
  \begin{equation}
    \rho_{\E,\ell}(\Frob_p^2)\varkappa(\Frob_p^2) =
    \begin{pmatrix}
      -s & -st \\
      t & t^2-s\end{pmatrix}.
  \end{equation}
  We claim that then (since $t = a_p(f) \neq 0$)
  \[
    \tilde{\rho}(\Frob_p) =
    \begin{pmatrix} 0 & -s\\ 1 & t
    \end{pmatrix}.
  \]
  The reason is that if $A$ is any $2 \times 2$ matrix with different
  eigenvalues, and $B$ is another $2 \times 2$ matrix satisfying that
  \[
    A^2 = B^2, \quad \trace(A) = \trace(B) \neq 0,
  \]
  then $A = B$ (which follows from an elementary computation, assuming
  that $A$ is diagonal). This implies that the representation
  $\tilde{\rho}$ can be chosen to take values in
  $\GL_2(L_\lambda)$. In particular, for any prime $q$ not dividing
  the level of $f$,
  $\trace(\tilde{\rho}(\Frob_q)) = a_q(f) \in L_\lambda$ for all
  primes $\lambda \in L$, hence $a_q(f) \in L$ (by
  Lemma~\ref{lemma:stupid}) and the same is true for primes dividing
  the level of $f$ (by Chebotarev density theorem). Since $K_f$ is the
  smallest field containing $a_q(f)$, we conclude that
  $K_f \subset L$.
\end{proof}
	
\begin{remark}
  The first case of the last result can occur. For example, let $E$ be
  a rational elliptic curve attached to a rational modular form $f$,
  and let $\varkappa$ be any quadratic character of $K$ which does not
  come from $\Q$. Let $\tilde{E}:= E \otimes \varkappa$ be the twist
  of $E$ by $\varkappa$. Then the coefficient field of $\tilde{E}$
  equals $K_f = \Q$, which is the trivial extension of
  $\Q(\zeta_2) = \Q$.
\end{remark}
	
\begin{lemma}
  \label{lemma:stupid}
  Let $\alpha \in \overline{\Q}$ and let $L$ be a number
  field. Suppose that for all prime ideals $\id{p}$ of $L$,
  $\alpha \in L_{\id{p}}$. Then $\alpha \in L$.
\end{lemma}
	
\begin{proof}
  Suppose that $\alpha \not \in L$, so $L(\alpha)/L$ is a non trivial
  extension. Let $\id{p}$ be a prime ideal of $L$ and $\id{q}$ be a
  prime ideal of $L(\alpha)$ dividing $\id{p}$ satisfying that the
  inertial degree $f(\id{q}\mid\id{p})$ is not $1$ (such an ideal
  always exists by applying Chebotarev's density theorem to the Galois
  closure of $L(\alpha)$ over $L$). Then $\alpha \not \in L_{\id{p}}$,
  contradicting the hypothesis.
\end{proof}
	
An important fact of the character $\varkappa$ (and also of
$\varepsilon$) is that by construction it has order a power of two
(although this is not explicitly stated in \cite{PT}, it follows from
its construction given in the proof of Theorem 3.2 in loc.\ cit).
	
\begin{lemma}
  \label{lemma-abelian-ext}
  Suppose that $\Q(\zeta_\ordkappa) \subsetneq K_f$. Then the
  extension $K_f/\Q$ is an abelian Galois extension. Furthermore, the
  field $K_f$ is the compositum of a quadratic extension of $\Q$ with
  $\Q(\zeta_\ordkappa)$.
\end{lemma}
\begin{proof}
  As proved in the last proposition, the quadratic extension
  $K_f/\Q(\zeta_\ordkappa)$ is obtained by adding the coefficient
  $a_p(f)$ for $p$ a prime that is inert in $K$ and of ordinary
  reduction for $\E$ satisfying that $a_p(f) \neq 0$. Recall that if
  $p$ is an inert prime, then $\varkappa^2(p) = \varepsilon(p^2)$, so
  $\varkappa(p) = \pm \varepsilon(p)$. Replacing this equality in
  (\ref{eq:inert}), we get that
  \begin{equation}
    \label{eq:inert-coef}
    a_p(f)^2 = \varepsilon(p) (\pm a_p(\E) + 2p).    
  \end{equation}
  Keeping the previous notation, let $\ordkappa$ be the order of
  $\varkappa$ (a power of $2$). The order of $\varepsilon$ equals the
  degree of the extension $[\overline{\Q}^{\varepsilon}:\Q]$. Since
  the character $\varepsilon$ is even (as proven in \cite[Section
  \textsection 3.1]{PT}), its fixed field is a totally real number
  field, so $\overline{\Q}^{\varepsilon} \cap K = \Q$. In particular,
  $[\overline{\Q}^{\varepsilon}:\Q] =[K\cdot
  \overline{\Q}^{\varepsilon}:K]=\frac{\ordkappa}{2}$ by
  Lemma~\ref{lemma:field-diagram}. Then $\varepsilon(p)$ is a root of
  unity of order a divisor of $\frac{\ordkappa}{2}$, hence a square in
  $\Q(\zeta_\ordkappa)$, so
  $K_f = \Q(\zeta_\ordkappa)[\sqrt{(\pm a_p(\E)+ 2p}]$ as claimed.
\end{proof}
	
The last lemma implies that if $\Q(\zeta_\ordkappa) \subsetneq K_f$,
then the Galois group $\Gal(K_f/\Q)$ is isomorphic to
$\Z/2 \times (\Z/\ordkappa)^\times$. It turns out that (in our
situation) all elements of such Galois group correspond a inner
twists.
	
\begin{thm}
  \label{thm:inner-twists}
  Let $\ordkappa$ be the order of the character $\varkappa$ and let
  $\delta_K$ denote the quadratic Dirichlet character corresponding to
  the extension $K/\Q$. Write $K_f = \Q(\zeta_\ordkappa)\cdot F$,
  where $F/\Q$ is at most a quadratic extension, as in the previous
  lemma. Let $j \in (\Z/\ordkappa)^\times$ and
  $\sigma_j \in \Gal(\Q(\zeta_M)/\Q)$ be the map given by
  $\sigma_j(\zeta_\ordkappa)=\zeta_\ordkappa^j$.  Then all inner
  twists of $A_f$ are the following:
  \begin{itemize}
  \item If $\sigma_j$ acts trivially on $F$, then
    $(\sigma_j,\varepsilon^{(j-1)/2})$ is an inner twist.
			
  \item If $\sigma_j$ does not act trivially on $F$, then
    $(\sigma_j,\delta_K\varepsilon^{(j-1)/2})$ is an inner twist.
  \end{itemize}
\end{thm}
\begin{proof}
  Let $p$ be a rational prime not dividing the level of $f$. If $p$
  splits in $K/\Q$, let $\id{p}$ a prime ideal of $K$ dividing
  $p$. Then $a_p(f) = \varkappa(\id{p}) a_{\id{p}}(\E)$, so
  \begin{equation}
    \label{eq:in1}
    \sigma_j(a_p(f)) = \varkappa^j(\id{p}) a_{\id{p}}(\E) = \varkappa^{j-1}(\id{p})(\varkappa(\id{p})a_{\id{p}}(\E)) = \varepsilon^{(j-1)/2}(p)a_p(f),
  \end{equation}
  where the last equality comes from the fact that
  $\varkappa^{(j-1)}(\id{p}) = (\varkappa^2(\id{p}))^{(j-1)/2}$,
  because $j$ is odd (recall that $\ordkappa$ is a power of $2$).  On
  the other hand, if $p$ is inert in $K$, it is enough to study the
  case when $a_p(f) \neq 0$. By~(\ref{eq:inert-coef}) we have
  \[
    a_p(f)^2 = \varepsilon(p)(\pm a_p(\E) + 2p).
  \]
  To ease notation, let $\eta =\pm a_p(\E) + 2p$. Note that
  $\varepsilon(p) = \zeta_\ordkappa^{2r}$ for some $r$ (see the proof
  of the last lemma), so $a_p(f) = \zeta_\ordkappa^r\sqrt{\eta}$ (for
  the right choice of the square root). Applying $\sigma_j$ to it we
  get
  \begin{equation}
    \label{eq:inert-case}
    \sigma_j(a_p(f)) = \zeta_\ordkappa^{jr}\sigma_j(\sqrt{\eta}) =
    a_p(f) \zeta_\ordkappa^{(j-1)r} \frac{\sigma_j(\sqrt{\eta})}{\sqrt{\eta}} =  a_p(f) \varepsilon(p)^{(j-1)/2}\frac{\sigma_j(\sqrt{\eta})}{\sqrt{\eta}}.
  \end{equation}
		
  If $\sigma_j(\sqrt{\eta}) = \sqrt{\eta}$ (i.e.\ $\sigma_j$ acts
  trivially on $F$), then equations~(\ref{eq:in1}) and
  (\ref{eq:inert-case}) imply that $(\sigma_j,\varepsilon^{(j-1)/2})$
  is an inner twist, while if $\sigma_j(\sqrt{\eta}) = -\sqrt{\eta}$
  (i.e.\ $\sigma_j$ does not act trivially on $F$), then both
  equations imply that $(\sigma_j,\delta_K \varepsilon^{(j-1)/2})$ is
  an inner twist.
\end{proof}
	
As an immediate application, using Lemma~\ref{lemma:fielddefinition},
we get the following result.
	
\begin{coro}
  All the endomorphisms of $A_f$ are defined over the field
  $K \cdot \overline{\Q}^{\varepsilon}$.
  \label{coro:endo-definition}
\end{coro}
	
Let $M_g$ be the constant coming from Theorem~\ref{thm:main}.
	
\begin{thm}
  \label{thm:main-decomposition}
  In the previous notation, and under the assumption that $K_f = K_g$
  and that $p > M_g$, there exists a building block $E_g$ defined over
  the quadratic field $K$ such that
  \[
    \bar{\rho}{_{\E,p}} \simeq \bar{\rho}{_{E_g,p}}.
  \]
\end{thm}
	
\begin{proof}
  Since all endomorphisms of $A_f$ are defined over
  $L:=K \cdot \overline{\Q}^\varepsilon$ (by
  Corollary~\ref{coro:endo-definition}), Theorem~\ref{thm:main}
  jointly with Proposition \ref{prop:decomposition} implies that
  $\End^0_L( A_g) \simeq M_r(\Q)$, where $r = \dim(A_f) =
  \dim(A_g)$. Then, over $L$ both varieties are isogenous to
  $r$-copies of an elliptic curve. Let us explain in more detail the
  situation for $A_f$. The extension $L/K$ is an abelian extension of
  order $M/2$, a power of $2$.  Over $\overline{\Q}$, $\E$ is a
  building block of $A_f$, but it is not a factor of its splitting over $L$ as
  we now explain. By Eichler-Shimura construction,
  \[
    \rho_{A_f,p}\simeq\bigoplus_{\sigma\in\Gal(K_f/\Q)}
    \rho_{\sigma(f),\id{p}},
  \]
  hence a similar decomposition holds while restricted to
  $\Gal_K$. For $\sigma$ the identity,
  $\rho_{f,\id{p}}|_{\Gal_K} \simeq \rho_{\E} \otimes \varkappa$,
  hence for any $\sigma \in \Gal(K_f/\Q)$ we have that
  \[
    \rho_{\sigma(f),\id{p}} \simeq \rho_{\E,\id{p}} \otimes
    \sigma(\varkappa).
  \]
  Recall that $\varkappa$ has order $\ordkappa$ and
  $\zeta_\ordkappa \in K_f$, hence while $\sigma$ ranges over all
  elements of $\Gal(K_f/\Q)$, $\sigma(\varkappa)$ ranges over all
  conjugates of $\varkappa$, which equals all its odd powers.
		
  Let $M'$ be the order of $\Gal(K_f/\Q)$.  By
  Proposition~\ref{prop:coef-field}, $M'$ equals $M$ or $M/2$,
  depending on whether $K_f$ is a quadratic extension of $\Q(\zeta_M)$
  or not. Then we get the following decomposition
  \begin{equation}
    \label{eq:decomposition-Af}
q			\rho_{A_f,p}|_{\Gal_K} \simeq  \left(\bigoplus_{i=1}^{M/2}  \rho_{\E,p}\otimes \varkappa^{2i-1}\right)^{2M'/M}.
\end{equation}  
The key point is that $\varkappa$ restricted to $\Gal_L$ is a
quadratic character (by Lemma~\ref{lemma:field-diagram}), so while
restricted to $\Gal_L$, the representation is isomorphic to $r$-copies
of $\rho_{\E,p}\otimes \varkappa$. In particular, the variety $A_f$
over $L$ is isogenous to $(\E \otimes \varkappa)^r$. The problem is
that the elliptic curve $\E \otimes \varkappa$ is not defined over
$K$. For that purpose, we look at the variety $A_f$ over
$\overline{K}^\varkappa$, and it is true that over such a field, $A_f$
is isogenous to $(\E)^r$ (the base change of a curve defined over
$K$). In particular, the cocycle $c$ attached to $\E$ (in the proof of
Theorem~\ref{thm:fielddefinition}) over $\overline{K}^\varkappa$ is
trivial while restricted to $\Gal_K$.
		
By Theorem~\ref{thm:fielddefinition} (and
Remark~\ref{remark:building-block}), there exists a building block
$E_g$ of $A_g$ defined over $K$ such that $A_g$ is isogenous (over
$\overline{K}^\varkappa$) to $E_g^r$.  In particular, the
semisimplification of the residual Galois representations
$\bar{\rho}{_{\E,p}}$ and $\bar{\rho}{_{E_g,p}}$ are isomorphic while
restricted to $\Gal_{\overline{K}^\varkappa}$.
		
The extension $\overline{K}^\varkappa/K$ is abelian, and the
characters of $\Gal(\overline{K}^\varkappa/K)$ are precisely powers of
$\varkappa$. Since $E_g$ is defined over $K$, we have that
\begin{equation}
  \label{eq:decomposition}
  \Ind_{\Gal_{\overline{K}^\varkappa}}^{\Gal_K}(\rho_{E_g ,p}|_{\Gal_{\overline{K}^\varkappa}}) \simeq \bigoplus_{i=1}^{M}\rho_{E_g,p}\otimes \varkappa^{i}.
\end{equation}
Since the curve $\E$ is also defined over $K$, a similar splitting
holds for
$\Ind_{\Gal_{\overline{K}^\varkappa}}^{\Gal_K}(\rho_{\E
  ,p}|_{\Gal_{\overline{K}^\varkappa}})$. Then
\[
  \bigoplus_{i=1}^{M}\bar{\rho}{_{E_g,p}}\otimes
  \overline{\varkappa}^{i} \simeq
  \bigoplus_{i=1}^{M}\bar{\rho}{_{\E,p}}\otimes
  \overline{\varkappa}^{i}.
\]
Note that since $\bar{\rho}{_{\E,p}}$ is absolutely irreducible, the
same must hold for $\bar{\rho}{_{E_g,p}}$. In particular,
$\bar{\rho}{_{\E,p}}$ must be a summand of the left-hand side, i.e.
\[
  \bar{\rho}{_{\E,p}} \simeq \bar{\rho}{_{E_g,p}}\otimes
  \overline{\varkappa}^{i},
\]
for some exponent $i$. Taking determinants on both sides, it follows
that either $\overline{\varkappa}^{i} =1$, or
$\overline{\varkappa}^{i}$ is a quadratic character. Note that since
$p$ is odd, and $\varkappa$ has order a power of two, $\varkappa$ and
$\overline{\varkappa}$ have the same order, so either $\varkappa^i$ is
trivial, or it is a quadratic character.  If $\varkappa^{i} = 1$ then
the result follows, while if $\varkappa^{i} \neq 1$, then the elliptic
curve $E_g \otimes \varkappa^{i}$ is another building block defined
over $K$ satisfying the required property.
\end{proof}
	
\subsection{The Diophantine equation $x^4+dy^2=z^p$}
Let us start with a general result on non-existence of elliptic curves
over quadratic fields with a $2$-torsion point.
	
\begin{thm}
  Let $d$ be a positive integer larger than $3$ such that the field
  $K=\Q(\sqrt{-d})$ satisfies the following properties:
  \begin{itemize}
  \item The prime $2$ is inert in $K/\Q$ $($i.e.
    $d \equiv 3 \pmod 8)$,
			
  \item The class number of $K$ is prime to $6$.
  \end{itemize} Then the only elliptic curves defined over
  $K=\Q(\sqrt{-d})$ having a $K$-rational point of order $2$ and
  conductor supported at the prime ideal $2$ are those that are base
  change of $\Q$.
  \label{thm:nonexistence}
\end{thm}
\begin{proof}
  Let $E/K$ be an elliptic curve satisfying the hypothesis. Since $K$
  does not have (in general) trivial class group, there is no reason
  for the curve $E$ to have a global minimal model. However, it always
  has what is called a ``semi-global'' minimal model, i.e.\ a model
  which is minimal at all primes dividing the conductor of the curve,
  but is not minimal at most one extra prime $\id{p}$ (which we assume
  is odd), and the discriminant valuation at $\id{p}$ equals $12$ (see
  \cite[Exercise 8.14]{MR2514094}). Our assumption that $E$ has only
  bad reduction at the prime ideal $(2)$ implies that
  $\Disc(E) = 2^r\id{p}^{12}$. In particular, $\id{p}^{12}$ is a
  principal ideal. Our assumption that the class number of $K$ is
  prime to $6$ then implies that $\id{p}$ is principal, hence $E$ does
  have a global minimal model.
		
  Take a global minimal model for the curve $E$. Doing the usual
  change of variables $y \to y-\frac{a_1x+a_3}{2}$ and clearing
  denominators, we can assume that our curve $E$ is given by a model
  (which might not be minimal at $2$ but is minimal at all other prime
  ideals) with $a_1=a_3=0$. Furthermore, we can assume (after a
  translation, which preserves minimality of the model at all odd
  primes) that our $2$-torsion point corresponds to the point $(0,0)$.
  Then the curve $E$ is given by an equation of the form
  \[
    E:y^2=x^3+ax^2+bx,
  \]
  where $a,b\in \ZZ[\frac{1+\sqrt{-d}}{2}]$. Minimality at all odd
  primes implies in particular that its discriminant
  $\Delta(E)=2^4b^2(a^2-4b)$ is a power of the prime ideal $(2)$,
  i.e.\ $\Delta(E)=(2)^r$, for some $r\ge0$. The hypothesis
  $K\neq \Q(\sqrt{-3})$ implies that the only roots of unity in $K$
  are $\pm 1$. Then
  \begin{equation}
    \label{eq:disc}
    b^2(a^2-4b)=\pm 2^{r-4}.    
  \end{equation}
  Since $K$ is a Dedekind domain, it has unique factorization in prime
  ideals, so in particular
  \[
    b=\pm 2^t,
  \]
  for some $t\ge 0$, and in particular $b\in\Z$ and since $\Disc(E)$
  is an algebraic integer, $2t+4 \le r$. Substituting in
  (\ref{eq:disc}) we get that
  \begin{equation}
    \label{eq:tildea}
    a^2=\pm 2^{r-4-2t}\pm 2^{t+2}\in\Z.  
  \end{equation}
  Suppose that $a=\frac{a_1+a_2\sqrt{-d}}{2}$, with $a_1, a_2 \in\ZZ$
  (and $a_1 \equiv a_2 \pmod 2$). Since $a^2$ is a rational number,
  $a_1a_2=0$. If $a_2=0$ then both $a,b$ are rational numbers and
  hence $E$ is a rational elliptic curve as claimed.
		
  Suppose then that $a_2\neq0$ and $a_1=0$, i.e.\ $a=a_2\sqrt{-d}$ for
  some integer $a_2$. Write $a_2$ in the form
  \[
    a_2=2^s\tilde{a},
  \]
  where $s\ge0$ and $2\nmid\tilde{a}$. Substituting in
  (\ref{eq:tildea}) we obtain the equation
  \[
    -d\tilde{a}^2=\pm2^{r-4-2t-2s}\pm2^{t+2-2s},
  \]
  where the exponents are non-negative integers and at least one of
  them must be zero (as the left-hand side is odd). The left-hand side
  is a negative integer which is congruent to $5$ mod $8$.  All
  solutions of the equation $\pm 1 \pm 2^m \equiv 5\pmod 8$ are
  \[
    \begin{cases}
      1 \pm 2^2 &\equiv 5 \pmod 8\\
      -1-2 &\equiv5\pmod 8.
    \end{cases}
  \]
  Note that in both cases, the non-zero exponent is at most $2$, so
  $d$ is at most $3$, which contradicts our assumption $d > 3$.
\end{proof}
	
\begin{remark}
  The result is not true over $\Q(\sqrt{-3})$, since for example the
  curve~\lmfdbecnf{2.0.3.1}{4096.1}{a}{1} has conductor $2^6$, a
  $2$-torsion point, but is not defined over the rationals (nor is
  isogenous to a rational elliptic curve). It is however a $\Q$-curve.
\end{remark}
	
\begin{remark} The last result is similar to \cite[Theorem
  1]{MR4054049}, in the case $\ell=2$, although in such an article the
  authors impose to the curve the condition that it has multiplicative
  reduction at $\ell$ (while our curve has additive reduction). In
  particular the condition on the class group being odd is the natural
  one (which matches theirs). The method of proof is completely
  different though.
\end{remark}
	
\begin{remark}
  The hypothesis on the class group being odd is equivalent to $d$
  being a prime number (under the assumption $d \equiv 3 \pmod
  8$). This was already discovered by Gauss (see for example \cite[\S
  6]{MR3236783}). The Cohen-Lenstra heuristics (\cite[(C2)]{MR756082}
  and also page 58 of loc.\ cit.) imply that the number of imaginary
  quadratic fields of prime discriminant, where $2$ is inert, and
  whose class group is not divisible by $3$, should have density
  $56.013\%$ (so there should be many of them).
\end{remark}
	
\begin{lemma}
  Let $E_1$ and $E_2$ be two elliptic curves over a number field
  $K$. Let $\id{q}$ be a prime of $K$ of good reduction for $E_1$. Let
  $p >\max\{\norm(\id{q})+1+2\sqrt{\norm(\id{q})},4\norm(\id{q})\}$ be
  a prime number such that
  $\bar{\rho}{_{E_1,p}} \simeq \bar{\rho}{_{E_2,p}}$. Then $E_2$ also
  has good reduction at $\id{q}$ and
  $a_{\id{q}}(E_1) = a_{\id{q}}(E_2)$.
  \label{lemma:equality-ap}
\end{lemma}
	
\begin{proof}
  The proof is similar to that of \cite[Theorem 1.4]{GPV}. Since
  $p>3$, the curve $E_2$ must have either good or multiplicative
  reduction at $\id{q}$ (by \cite[Remark 7]{GPV}). The reason is that
  if $E_2$ has additive reduction then its local type is either: a
  principal series, a ramified quadratic twist of an elliptic curve
  with multiplicative reduction or supercuspidal. A ramified quadratic
  character reduces modulo an odd prime to a ramified quadratic
  character, so the second case cannot occur.  Since the coefficient
  field of an elliptic curve is the rational one, any character (in
  both the principal series or the supercuspidal type) appearing at a
  prime of bad reduction has order $n$ with $\phi(n) \le 2$ (where
  $\phi(n)$ denotes Euler's totient function), so
  $n \in \{1,2,3,4, 6\}$. Then for $p>3$, the local type of $\rho$ is
  preserved by congruences.
		
  If the reduction is multiplicative, we are in a ``lower the level''
  case, hence
  \begin{equation}
    \label{eq:lowering-the-level1}
    a_{\id{q}}(E_1) \equiv \pm(\norm(\id{q})+1) \pmod{p}.    
  \end{equation}
  But Hasse's bound implies that
  $\abs{a_{\id{q}}(E_1)}\le 2\sqrt{\norm(\id{q})}$. Since the
  difference of the right and the left-hand side of
  (\ref{eq:lowering-the-level1}) is non-zero, then $p$ must be smaller
  than their difference, which contradicts the hypothesis
  $p > 1+2\sqrt{\norm(\id{q})}+\norm(\id{q})$.
		
  Once we know that both curves have good reduction at $\id{q}$, we
  get the congruence
  \[
    a_{\id{q}}(E_1) \equiv a_{\id{q}}(E_2) \pmod p.
  \]
  If both numbers are different, $p$ must divide their difference,
  which by Hasse's bound is at most $4\norm(\id{q})$, giving the
  result.
\end{proof}
	
\begin{lemma}
  \label{lemma:not-div}
  Let $(a,b,c)$ a primitive solution of~(\ref{eq:42p}), where $p$ is
  large enough so that $\bar{\rho}{_{\E,p}}$ is absolutely irreducible
  and let $q$ an odd prime. There exists a bound $B$ (depending only
  on $d$ and on $q$) such that if $p > B$ then $q\nmid c$.
\end{lemma}
	
\begin{proof}
  If $q\mid c$ then the curve $\E$ has multiplicative reduction at
  $q$, but it does not divide the conductor of the residual
  representation $\bar{\rho}{_{\E,p}}$. In particular, the form $f$
  has level divisible by $q$, but the form $g$ does not, i.e.\ we are
  in what is called the ``lower the level case". In particular,
  \begin{equation}
    \label{eq:lowering-the-level}
    p\mid \norm(\varepsilon^{-1}(q)(q+1)^2-a_q(g)^2).    
  \end{equation}
  Note that $\varepsilon$ depends only on $d$ and there are finitely
  many possibilities for the value $a_q(g)$ (since the form $g$ is a
  newform in the space $S_2(\Gamma_0(\tilde{N}), \varepsilon)$ which
  does not depend on $(a,b,c)$), so it is enough to prove that the
  right-hand side of~(\ref{eq:lowering-the-level}) is non-zero for any
  newform $g$. But by the Ramanujan-Petersson conjecture,
  $|a_q(g)|^2 \le 4q < (q+1)^2$, so the difference cannot be zero.
\end{proof}

\begin{thm}
  Let $d$ be a prime number congruent to $3$ modulo $8$ and such that
  the class number of $K=\Q(\sqrt{-d})$ is not divisible by $3$. Then
  there are no non-trivial primitive solutions of the equation
  \[
    x^4+dy^2=z^p,
  \]
  for $p$ large enough.
  \label{thm:asymptotic}
\end{thm}
\begin{proof}
  The case $d=3$ was proven in \cite{MR2561200}, so we can restrict to
  values $d >3$. Let $(a,b,c)$ be a non-trivial primitive solution and
  consider the elliptic curve $\E$ as in (\ref{eq:curve-E}). The
  assumption $(a,b,c)$ being non-trivial and primitive implies that
  $\E$ does not have complex multiplication (by
  Lemma~\ref{lemma:no-CM}).
		
  The discriminant of $\E$ equals $512(a^2+b\sqrt{-d})c^p$, which is a
  perfect $p$-power except at the prime $2$. Furthermore, all odd
  primes dividing the conductor of $\E$ are of multiplicative
  reduction, so the residual representation $\bar{\rho}{_{\E,p}}$ is
  unramified outside $2$.
		
  Recall that there exists a Hecke character $\varkappa$ and a newform
  $f \in S_2(\Gamma_0(N),\varepsilon)$ (where $N$ depends on
  $(a,b,c)$, but the conductor of $\varepsilon$ is supported only at
  primes dividing $2d$) such that the twisted representation
  $\rho_{\E,p}\otimes \varkappa$ extends to $\rho_{f,\id{p}}$.

  Note that if $(a,b,c)$ is a non-trivial solution of~\eqref{eq:42p},
  clearly $c \neq 1$.  Looking at~(\ref{eq:42p}) modulo $8$ it follows
  that $c$ is not divisible by $2$, and by Lemma~\ref{lemma:not-div},
  we can assume that $c$ is not divisible by $3$. Then there exists a
  prime number $q$ larger than $3$ such that $c$ is divisible by
  $q$. In particular, the curve $\E$ has a prime of multiplicative
  reduction not dividing $6$. Ellenberg's large image result
  (\cite[Theorem 3.14]{MR2075481}) implies that there exists a bound
  $B$ (depending only on $K$) such that if $p>B$, then the projective
  residual image of $\rho_{\E,p}$ is surjective.
		
  In particular, for such primes $p$ we are in the hypothesis of
  Ribet's lowering the level result, so there exists a newform
  $g \in S_2(\Gamma_0(\tilde{N}), \varepsilon)$ such that the residual
  representation of $\rho_{g,\id{p}}$ is isomorphic to that of
  $\tilde{\rho}$, where $\tilde{N}$ is only divisible by primes
  dividing $2d$. The surjectivity of the residual representation of
  $\rho_{\E,p}$ implies that the form $g$ cannot have complex
  multiplication either (which justifies such an hypothesis made in
  the present article) so we can apply our previous results.
		
  Let $g$ be any newform in the space
  $S_2(\Gamma_0(\tilde{N}), \varepsilon)$ without complex
  multiplication and suppose that it is related to a solution
  $(a,b,c)$ of (\ref{eq:42p}). Let $g^{\text{BC}}$ denote its base
  change to $K$. Let $\ell$ be a prime not dividing $2d$ and define
  \[
    S_\ell = \{(\tilde{a},\tilde{b},\tilde{c}) \in
    \F_\ell^3\setminus\{(0,0,0)\} \; : \;
    \tilde{a}^4+d\tilde{b}^2=\tilde{c}^p\}.
  \]
  In practice, since $p$ will be a larger prime (compared to $\ell$),
  raising to the $p$-th power is a bijection of $\F_\ell$.  For each
  point $(\tilde{a},\tilde{b},\tilde{c}) \in S_\ell$, consider the
  curve $E_{(\tilde{a},\tilde{b},\tilde{c})}$ over $\F_\ell$. Let
  $\id{l}$ be a prime of $K$ dividing $\ell$. Then, either:
  \begin{enumerate}
  \item The curve $E_{(\tilde{a},\tilde{b},\tilde{c})}$ is
    non-singular, in which case if $(a,b,c)$ is an integral solution
    reducing to $(\tilde{a},\tilde{b},\tilde{c})$, we must have that
    $a_{\id{l}}(\E) = a_{\id{l}}(E_{(\tilde{a},\tilde{b},\tilde{c})})$
    and furthermore
    \[
      \varkappa(\id{l})
      a_{\id{l}}(E_{(\tilde{a},\tilde{b},\tilde{c})}) \equiv
      a_{\id{l}}(g^{\text{BC}}) \pmod p,
    \]
    or
  \item The curve $\E$ has bad reduction at $\id{l}$ in which case we
    are in the lowering the level hypothesis, and
    \[
      a_\ell(g) \equiv \pm \varkappa^{-1}(\id{l}) (\ell+1) \pmod p.
    \]
  \end{enumerate}
		
  Given $(\tilde{a},\tilde{b},\tilde{c}) \in S_\ell$, let
  $B(\ell,g;\tilde{a},\tilde{b},\tilde{c})$ be given by
  \[
    B(\ell,g;\tilde{a},\tilde{b},\tilde{c})=
    \begin{cases}
      \norm(a_{\id{l}}(\E)\varkappa(\id{l})-a_\ell(g)) & \text{ if } \ell \nmid \tilde{c} \text{ and } \ell \text{ splits in }K,\\
      \norm(a_\ell(g)^2-a_\ell(\E)\varkappa(\ell)-2\ell\varepsilon(\ell)) & \text{ if } \ell \nmid \tilde{c} \text{ and } \ell \text{ is inert in }K,\\
      \norm(\varepsilon^{-1}(\ell)(\ell+1)^2-a_\ell(g)^2)  & \text{ if }\ell \mid \tilde{c}.
    \end{cases}
  \]
  If $(\tilde{a},\tilde{b},\tilde{c})$ belongs to case $(2)$, then
  clearly $p \mid B(\ell,g;\tilde{a},\tilde{b},\tilde{c})$ (since
  $\varkappa(\id{l})^2=\varepsilon(\ell)$). If
  $(\tilde{a},\tilde{b},\tilde{c})$ belongs to case $(1)$, the well
  known formula for the Fourier coefficients of $g^{\text{BC}}$ in
  terms of those of $g$, imply that
  \[
    \begin{cases}
      \varkappa(\id{l})a_{\id{l}}(E_{(\tilde{a},\tilde{b},\tilde{c})}) \equiv a_{\ell}(g) \pmod p& \text{ if $\ell$ splits},\\
      \varkappa(\id{l})a_{\id{l}}(E_{(\tilde{a},\tilde{b},\tilde{c})}) \equiv a_\ell(g)^2-2\ell \varepsilon(\ell) \pmod p& \text{ if $\ell$ is inert}.
    \end{cases}
  \]
  In all cases, it holds that
  \begin{equation}
    \label{eq:mazur}
    p\mid \prod_{(\tilde{a},\tilde{b},\tilde{c}) \in
      S_\ell}B(\ell,g;\tilde{a},\tilde{b},\tilde{c}). 
  \end{equation}
		
  As previously explained, the Ramanujan--Petersson conjecture implies
  that the third row value in the definition of
  $B(\ell,g;\tilde{a},\tilde{b},\tilde{c})$ is never zero. If the
  coefficient field $K_g$ does not match the coefficient field $K_f$,
  then there exists some prime $\ell$ for which the first or the
  second row (depending on whether $\ell$ is split or inert) is
  non-zero, so the right-hand side of (\ref{eq:mazur}) is non-zero,
  giving finitely many possibilities for the value of the prime $p$
  (this is an idea of Mazur). Then to finish the proof we are left to
  discard the newforms $g$ whose coefficient field $K_g$ matches the
  coefficient field $K_f$ of $f$.
		
  By Theorem~\ref{thm:main-decomposition}, if $p$ is large enough,
  there exists an elliptic curve $E_g$ defined over $K$, whose
  conductor divides $\tilde{N}$ such that
  $\bar{\rho}{_{E_g,p}}\simeq \bar{\rho}{_{\E,p}}$. A priori, the
  curve $E_g$ has bad reduction at primes dividing $\tilde{N}$ (i.e.\
  at primes dividing $2d$), but the curve $\E$ has good reduction at
  all odd primes dividing $d$, so in particular the same must be true
  for the curve $E_g$ if $p>3$ (see~\cite[Proposition 1.1]{GPV} and
  also Remark 7).
		
  The residual representation $\bar{\rho}{_{E_g,2}}$ has image lying
  in $S_3$. Under the isomorphism $\GL_2(\F_2) \simeq S_3$, the
  elements of order $1$ or $2$ are precisely the ones of trace $0$,
  while the ones of order $3$ have trace $1$. In particular, the image
  of the residual representation (is isomorphic to one that) lies in
  the Borel subgroup (i.e.\ the curve $E_g$ has a $K$-rational point
  of order $2$) if and only if the trace of any Frobenius element is
  even if and only if the image does not have elements of order
  $3$. Let $T$ denote the fixed field of the kernel of
  $\bar{\rho}{_{E_g,2}}$, so the extension $T/K$ is unramified outside
  $2$ (by the Ner\'on-Ogg-Shafarevich criterion) and is of degree at
  most $6$. A well known result of Hermite and Minkowski states that
  there are finitely many field extensions of a given degree and
  bounded discriminant. In particular, our field $T$ is one of a
  finite list, say $\{T_1, \ldots, T_n\}$ of at most degree $6$
  extensions of $K$ unramified outside $\{2\}$. Suppose that $[T_i:K]$
  is divisible by $3$ for some index $i$. Then an explicit version of
  Chebotarev's density theorem (see for example~\cite{MR3671433} and
  the references therein) proves the existence of a bound $B$ and a
  prime $\id{q} \in \Om_K$ (the ring of integers of $K$) of norm at
  most $B$ such that $\Frob_{\id{q}}$ has order $3$ in
  $\Gal(T_i/K)$. In particular, if $[T:K]$ is divisible by $3$, there
  exists a prime whose norm is bounded by $B$ (independently of the
  original solution $(a,b,c)$) such that
  \[
    a_{\id{q}}(E_g) = \trace(\rho_{E_g,2}(\Frob_{\id{q}})) \equiv 1
    \pmod 2.
  \]
  But Lemma~\ref{lemma:equality-ap} implies that if $p$ is large
  enough (where the bound depends on the norm of the prime $\id{q}$,
  which is bounded by $B$) then $\E$ has good reduction at $\id{q}$
  and furthermore $a_{\id{q}}(E_g) = a_{\id{q}}(\E)$. Recall that $\E$
  has a $2$-torsion point, so $a_{\id{q}}(\E)$ is even, giving a
  contradiction. Then $T/K$ has degree $1$ or $2$ and the residual
  representation $\bar{\rho}{_{E_g,2}}$ has image in a Borel subgroup,
  so the curve $E_g$ also has a $K$-rational $2$-torsion point.
		
  Theorem~\ref{thm:nonexistence} then implies that the elliptic curve
  $E_g$ is in fact defined over $\Q$. In particular, if $q$ is a prime
  integer that splits in $K$ as $(q) = \id{q} \overline{\id{q}}$, then
  \[
    a_{\id{q}}(E_g) = a_{\overline{\id{q}}}(E_g).
  \]
  However, the curve $\E$ satisfies the property (proved in
  \cite[Proposition 2.2]{PT})
  \[
    a_{\id{q}}(\E) = \kro{-2}{q} a_{\overline{\id{q}}}(\E).
  \]
  Recall that by Ellenberg's result (\cite[Theorem 3.14]{MR2075481})
  we are assuming that the projective residual image of $\rho_{\E,p}$
  equals $\PGL_2(\F_p)$. Its fixed field is an extension of $K$
  disjoint from $K(\sqrt{-2})$ if $q$ is odd. Then by Chebotarev's
  theorem there exists a prime ideal $\id{q}$ of $K$ of prime norm $q$
  of good reduction for both $E_g$ and $\E$ such that $a_{\id{q}}(\E)$
  is not divisible by $p$ (if for example
  $\overline{\rho}_{\E,p}(\Frob_{\id{q}})$ is the identity matrix in
  $\PGL_2(\F_q)$) and $\kro{-2}{q}=-1$. This contradicts the fact that
  $E_g$ and $\E$ are congruent modulo $p$.
\end{proof}
	
\subsection{The Diophantine equation $x^2+dy^6=z^p$}
	
\begin{thm}
  \label{thm:asymptotic2}
  Let $d$ be a prime number congruent to $19$ modulo $24$ and such
  that the class number of $K=\Q(\sqrt{-d})$ is prime to $6$. Then
  there are no non-trivial primitive solutions of the equation
  \[
    x^2+dy^6=z^p,
  \]
  for $p$ large enough.
\end{thm}
\begin{proof}
  The proof mimics that of Theorem~\ref{thm:asymptotic}, in particular
  all results of Section~\ref{section:decomposition} hold for $\Et$,
  with the following important observation:
		
  (1) The curve $\Et$ does not have complex multiplication if $p>3$ by
  \cite[Lemma 3.2]{GPV}.
		
  (2) The curve $\Et$ has additive reduction at the prime $\sqrt{-d}$,
  and acquires good reduction over the extension $K(\sqrt[6]{-d})$
  (see~\cite[Remark 2]{PT}). In particular, if $\tilde{E}_g$ denotes
  the elliptic curve defined over $K$ that is obtained after applying
  the lowering the level result to $\Et$ (whose existence is warranted
  by Theorem~\ref{thm:main-decomposition}), then it also acquires good
  reduction over the extension $K(\sqrt[6]{-d})$, hence its minimal
  discriminant valuation at the prime $\sqrt{-d}$ must be even.
		
  (3) The curve $\Et$ has a $K$-rational $3$-torsion point, so we
  would like to know that the same is true for $\tilde{E}_g$. Since
  the curve $\Et$ has a point of order $3$, for all prime ideals
  $\id{p}$ of good reduction,
  $a_{\id{p}}(\Et) \equiv \norm(\id{p})+1 \pmod{3}$. Using
  Lemma~\ref{lemma:equality-ap} we know that
  $a_{\id{p}}(\Et) = a_{\id{p}}(\tilde{E}_g)$ for all small prime
  ideals $\id{p}$. In particular,
  $a_{\id{p}}(\tilde{E}_g) \equiv \norm(\id{p})+1 \pmod 3$ for all
  small prime ideals, hence by the so called ``Sturm'' bound (see for
  example Corollary 9.20 of \cite{MR2289048}), the congruence holds
  for all prime ideals of good reduction. Then by \cite[Theorem
  2]{MR604840} there exists a curve $E'$ over $K$ which is isogenous
  to $\tilde{E}_g$ over $K$ which has a rational point of order $3$.
		
  (4) The hypothesis $d\equiv 19 \pmod {24}$ implies that the primes
  $2$ and $3$ are inert in $K/\Q$. Then the curve $\Et$ has reduction
  type $\text{IV}^*$ at the prime ideal $(2)$ (by \cite[Lemma
  22.14]{PT}), so it has additive but potentially good reduction. Then
  its local type is preserved by a congruence modulo any prime larger
  than $3$ (as explained in the proof of
  Lemma~\ref{lemma:equality-ap}) so $\tilde{E}_g$ has also additive
  but potentially good reduction at the prime ideal $(2)$. The same is
  true for the prime ideal $(3)$, the curve $\Et$ has reduction type
  II or III (by \cite[Lemma 2.15]{PT}).
		
  Then there exists an elliptic curve $E'$ defined over $K$ with the
  following properties:
		
  \begin{itemize}
  \item The conductor of $E'$ is supported at the prime ideals
    dividing $6d$.
  \item If the model $E'$ is minimal at the prime ideal $(\sqrt{-d})$,
    then the discriminant $\Disc(E')$ of $E'$ has even valuation at
    $(\sqrt{-d})$.
  \item The curve $E'$ has a $K$-rational $3$-torsion point $P$.
			
  \item The curve $E'$ has potentially good reduction at $(2)$ and at
    $(3)$.
  \end{itemize}

  Take a semi-global minimal model for $E'$, i.e.\ a model which is
  minimal at all primes except one extra prime ideal $\id{p}$, which
  we can assume that does not divide $6d$. The coordinates of the
  point $P$ are algebraic integers (by \cite[Theorem 3.4, Chapter
  VII]{MR2514094}) so after a translation (a transformation which
  preserves the discriminant of the equation) we can assume that the
  rational $3$-torsion point is the origin $(0,0)$, so the model is of
  the form
  \begin{equation}
    \label{eq:4}
    E': y^2 + a_1xy + a_3y = x^3+a_2x^2+a_4x,    
  \end{equation}
  where $a_1, a_2,a_3,a_4$ are algebraic integers. Let $y=\alpha x$ be
  the tangent line of $E'$ at $P$ (for some number $\alpha$).  The
  fact that $P=(0,0)$ is an inflection point of $E'$ implies that the
  substitution $y=\alpha x$ on equation~(\ref{eq:4}) (and substracting
  the left hand side with the right hand side) gives the polynomial
  $-x^3$.

  Then $\alpha^2+a_1\alpha - a_2=0$, so $\alpha$ is an
  algebraic integer.  The change of variables $y' = y-\alpha x$,
  $x' = x$ (which preserves the discriminant and the properties of the
  model) sends the tangent line to the line $y'=0$. In particular, we
  can (and do) assume that our semi-global minimal model is of the
  form
  \[
    E':y^2+a_1 xy + a_3y=x^3,
  \]
  where $a_1,a_3\in \ZZ[\frac{1+\sqrt{-d}}{2}]$. In particular,
  \[
    \Disc(E') = a_3^3(a_1^3-27a_3)=2^r 3^q
    (\sqrt{-d})^{2s}\id{p}^{12}.
  \]
  The even exponent at $(\sqrt{-d})$ comes from the fact that the
  model is minimal at $(\sqrt{-d})$ and the second condition. In
  particular, the ideal $\id{p}^{12}$ is a principal ideal, so under
  our assumption on the class number of $K$ being prime to $6$,
  $\id{p}$ is principal and hence $E'$ does have a global minimal
  model (of the same form). In particular, since the only roots of
  unity in $K$ are $\pm 1$, for the minimal model it holds that
  \begin{equation}
    \label{eq:disc2}
    \Disc(E') = a_3^3(a_1^3-27a_3)=\pm 2^r 3^q d^{s}.    
  \end{equation}
  If $(\sqrt{-d})$ does not divide the gcd of the two middle factors
  (as elements of $K$), $a_3$ must be a rational number. Then $a_1^3$
  is also a rational number and hence $a_1$ is rational. On the other
  hand, if $(\sqrt{-d})$ divides the gcd of the two middle factors,
  then the minimality condition of the model $E'$ implies that
  $v_{(\sqrt{-d})}(a_3) \le 2$, so it is either $1$ or $2$. If it
  happens to be $2$, then $a_3$ is once again a rational number, and
  the same proof as before implies that $a_1$ is rational as well.
		
  Suppose then that $a_3 = \sqrt{-d} \cdot \beta$ for some algebraic
  integer $\beta$ not divisible by $(\sqrt{-d})$, and that
  $a_1 = \sqrt{-d}\cdot \alpha$ for some algebraic integer
  $\alpha$. Then the valuation at $(\sqrt{-d})$ of the middle term in
  (\ref{eq:disc2}) is $4$, hence $s=2$ and we get the equation
  \begin{equation}
    \label{eq:3}
    \beta^3(d\alpha^3+27\beta) = \pm 2^r 3^q.
  \end{equation}
  Once again, $\alpha$ and $\beta$ must be integers. Since the curve
  $E'$ has potentially good reduction at the prime ideals $(2)$ and
  $(3)$, there is a bound for $r$ and $q$ that we recall. By
  \cite[Theorem 10.4]{zbMATH00706265} the exponent conductor of $E/K$
  at the prime ideal $(2)$ is bounded by $8$ and at the prime ideal
  $(3)$ is bounded by $5$. By \cite[Table 4.1]{zbMATH00706265} (page
  365) the number of irreducible components of the special fiber of
  Neron's model is at most $9$. Then Ogg's formula \cite[Theorem
  11.1]{zbMATH00706265} implies that $r \le 16$ and $s \le 13$. Then
  we can run over all possible exponents on the right-hand side within
  this bound, and verify for which values, we get a divisor $\beta$
  such that $\pm 2^r3^q/\beta^3-27\beta$ is a prime times a perfect
  cube.  Furthermore, we discard the solutions for which the curve
  $E'$ does not have additive reduction at both primes $2$ and $3$
  (since the curve $\Et$ has this property). We get only four
  non-rational candidates, all of them defined over the quadratic
  field $K = \Q(\sqrt{-547})$, corresponding to the values
  \[
    (\alpha, \beta) \in \{(-6,2), (-12,16), (6,-2), (12, -16)\}.
  \]
  Clearly there are only two non-isomorphic pairs (the map
  $(x,y) \to (x,-y)$ gives an isomorphism between a pair $(a_1,a_3)$
  and a pair $(-a_1,-a_3)$). It is easy to verify that for the two
  isomorphism classes of curves, the quotient by the $3$-torsion point
  is a rational elliptic curve, hence the curve $E'$ is isogenous to a
  base change. As in Theorem~\ref{thm:asymptotic} this contradicts the
  fact that $\Et$ satisfies
  \[
    a_{\id{q}}(\Et) = \kro{-3}{\norm(\id{q})}
    a_{\overline{\id{q}}}(\Et),
  \]
  for all prime ideals $\id{q}\nmid 3$ of good reduction (as proven in
  \cite[Proposition 2.3]{PT}).
\end{proof}
	
\bibliographystyle{plain}
\bibliography{biblio}
	
\end{document}